\documentclass[11pt]{amsart}

\usepackage{upgreek}
\usepackage{longtable}
\usepackage{ytableau}
\usepackage{shuffle}

\usepackage{shuffle}
\usepackage{amsmath, amsxtra, amsthm, amssymb,mathtools,bm,amsfonts}
\usepackage{mathrsfs}
\usepackage[mathscr]{euscript}
\usepackage{graphicx}
\usepackage{url}
\usepackage{color}
\usepackage{bbm}
\usepackage{tikz-cd}
\usepackage{tikz}

\usepackage[T1]{fontenc}
\graphicspath{/Figures/}

\usepackage{enumitem}
\usepackage{comment}
\usepackage[pdftex,hidelinks,backref=page]{hyperref}
\hypersetup{
    colorlinks,
    citecolor=magenta,
    filecolor=magenta,
    linkcolor=blue,
    urlcolor=black
}
\usepackage{cleveref}

\usepackage{xcolor}

\renewcommand{\emptyset}{\varnothing}

\newcommand{\CC}{\mathbb C}

\newcommand{\ZZ}{\mathbb Z}

\newcommand{\plucker}[1]{\operatorname{Pl}_{#1}}

\theoremstyle{definition}
\newtheorem{thm}{Theorem}[section]
\newtheorem{cor}[thm]{Corollary}
\newtheorem{lem}[thm]{Lemma}

\newtheorem{prop}[thm]{Proposition}
\newtheorem{defn}[thm]{Definition}

\newtheorem{conj}[thm]{Conjecture}
\newtheorem{prob}[thm]{Problem}
\newtheorem{eg}[thm]{Example}
\newtheorem{rem}[thm]{Remark}

	\newtheorem{fact}[thm]{Fact}

\numberwithin{equation}{section}

\newcommand{\qsymide}[2][]{
{\ifx&#1&%
  {\operatorname{QSym}_{#2}^+}
\else
  {{}^{#1}\!\operatorname{QSym}_{#2}^+}
\fi}
} 
\newcommand{\sym}[1]{\operatorname{Sym}_{#1}} 
\newcommand{\symide}[1]{\sym{#1}^+} 
\newcommand{\lcode}[1]{\operatorname{lcode}(#1)} 
\newcommand{\suchthat}{\;|\;}

\newcommand{\schub}[1]{\mathfrak{S}_{#1}} 
\newcommand{\des}[1]{\operatorname{Des}(#1)} 

\date{}
\newcommand{\cat}[1]{\operatorname{Cat}_{#1}} 

\newcommand{\idem}{\operatorname{id}} 
\newcommand{\slide}[2][]{
{\ifx&#1&%
  {\mathfrak{F}_{#2}}
\else
  {\mathfrak{F}_{#2}^{\underline{#1}}}
\fi}
} 

\newcommand{\ct}{\operatorname{ev_0}} 
\newcommand{\rope}[1]{\mathsf{R}_{#1}} 

\newcommand{\fl}[1]{\mathrm{Fl}_{#1}}

\newcommand{\ins}{\varepsilon}

\newcommand{\xl}{\textbf{x}}

\newcommand{\ev}{\operatorname{ev}}

\newcommand{\mc}[1]{\mathcal{#1}}

\renewcommand\emph[1]{\textcolor{blue}{\textit{#1}}} 

\newcommand{\wellaligned}[1]{\operatorname{WA}_{#1}}

\newcommand{\syt}{\operatorname{SYT}} 
\newcommand{\crop}{\operatorname{crop}} 
\newcommand{\richtab}{\operatorname{RT}} 
\newcommand{\flag}[1]{\operatorname{Fl}_{#1}}
\newcommand{\rothe}{\mc{R}} 

\newcommand{\evac}[1]{\operatorname{evac}(#1)}
\newcommand{\First}{\operatorname{First}}
\newcommand{\GL}{\operatorname{GL}}
\newcommand{\cox}{\mathbf{c}} 

\title{Richardson tableaux and Schubert positivity}

\author{Hunter Spink}
\address{Department of Mathematics,
University of Toronto, Toronto, ON M5S 2E4, Canada}
\email{\href{mailto:hunter.spink@utoronto.ca}{hunter.spink@utoronto.ca}}

\author{Vasu Tewari}
\address{Department of Mathematical and Computational Sciences, University of Toronto Mississauga, Mississauga, ON L5L 1C6, Canada}
\email{\href{mailto:vasu.tewari@utoronto.ca}{vasu.tewari@utoronto.ca}}

\thanks{
HS and VT acknowledge the support of the NSERC, respectively [RGPIN-2024-04181] and [RGPIN-2024-05433].}


\begin{document}
\begin{abstract}
We compute the Schubert cycle expansion of those irreducible components of Springer fibers equal to Richardson varieties. This generalizes work of G\"uemes in the case of a hook shape and answers a question of Karp--Precup.
\end{abstract}

\maketitle

Let $B,B^-\subset \GL_n$ denote the Borel and opposite Borel subgroups of upper triangular and lower triangular matrices respectively, and let  $\mathfrak{b},\mathfrak{b}^-,\mathfrak{gl}_n$ denote their respective Lie algebras. Karp and Precup \cite{KaPr25} recently studied the interplay between two distinguished classes of subvarieties of the complete type $A$ flag variety $\flag{n}=\GL_n/B$. 
\begin{enumerate}
    \item The first class are the \emph{Springer fibers} $$\mc{B}_\lambda=\{gB\in \flag{n}\suchthat g^{-1}Mg\in \mathfrak{b}\},$$ for $M\in \mathfrak{gl}_n$ a fixed nilpotent operator of Jordan type the partition $\lambda\vdash n$ \cite{Spr69}. The irreducible components are indexed by standard Young tableaux (SYTs) of shape $\lambda$ \cite{Spa76}, and 
Springer \cite{Spr76} showed the homology of Springer fibers carry an action of the symmetric group which is irreducible in the top dimension, connecting the geometric study of Springer fibers to representation theory.
\item The second class are the \emph{Richardson varieties} $$
X_{v}^w\coloneqq \overline{BwB/B}\cap \overline{B^-vB/B}\subset \fl{n},
$$ indexed by pairs of permutations $(v,w)$ where $v\leq_B w$ in Bruhat order.
These varieties are irreducible, and are obtained by intersecting Schubert varieties $X^w=\overline{BwB/B}\subset \fl{n}$ with opposite Schubert varieties $X_v=\overline{B^-vB/B}\subset \fl{n}$. Richardson varieties play a crucial role in the connection between algebraic geometry and Schubert calculus because of the identity
$$[X^w_v]=\sum_u c^w_{u,v}[X^u]\in H_\bullet(\flag{n})$$
which decomposes the homology class of the Richardson variety into the homological basis of Schubert varieties, with coefficients $c^w_{u,v}$ the structure constants for Schubert polynomial multiplication $\schub{u}\schub{v}=\sum_w c^w_{u,v}\schub{w}$. 
Showing the nonnegativity of $c^w_{u,v}$ via combinatorial means for arbitrary triples $(u,v,w)$ is an outstanding open problem in algebraic combinatorics \cite[Problem 11]{St00}. 
The answer, outside of the celebrated Littlewood--Richardson rule for Grassmannian permutations \cite{KT99,Li40,Sch77}, is known in special cases. See for instance \cite{gaozhu,Hu23,huang2022bumpless, KZJ3, PW24} for manifestly nonnegative rules, \cite{PC25signed} for a signed rule, and \cite{PC25} for a complexity-theoretic perspective.
\end{enumerate}
Motivated by Lusztig's study \cite{Lus21} of the totally nonnegative Springer fiber, Karp and Precup classified the SYTs which they called \emph{Richardson tableaux}  that index irreducible components of Springer fibers which are equal to Richardson varieties \cite[Theorem 1.5]{KaPr25}.

It is unknown what the coefficients of the expansion of arbitrary Springer fibers and their irreducible components into Schubert cycles are in full generality, and this question was raised by Springer \cite[Problem 4]{opag83}. 
For tableaux of hook shape the decomposition was determined by G\"uemes \cite{Gu89}, and all such tableaux are Richardson tableaux \cite[\S 10]{KaPr25}. 
In \cite[Problem 10.4]{KaPr25} Karp and Precup ask Springer's question in the special case of a Richardson tableau $T$ -- whether there is a combinatorial way to compute the Schubert cycle expansion of the corresponding Springer fiber component/Richardson variety $X_{v_T}^{w_T}$. 
We answer this question:
\begin{thm}
  Given a Richardson tableau $T$, there is a combinatorial rule for computing $c^{w_T}_{u,v_T}$.
\end{thm}
We note further that the components of Springer fibers studied by Graham--Zierau \cite{GZ11} are a strict subset of those corresponding to Richardson tableaux (see \cite[Section 11.3]{KaPr25}). 
So the aforementioned theorem gives a combinatorial alternative to localization formulae derived in \cite{GZ11}.

Our strategy is to describe a family of pairs of permutations $(v,w)$ we call \emph{well-aligned}, for which we can combinatorially compute the Schubert structure coefficients $c^w_{u,v}$, and show that $(v_T,w_T)$ are well-aligned when $T$ is a Richardson tableau. 
For pairs of well-aligned permutations we give two distinct combinatorial ways of computing these coefficients. \begin{enumerate}
    \item One way is to show that well-aligned pairs give Bruhat intervals $[v,w]$ that are  translation-equivalent to $[v',w']$ with $v'$ a $132$-avoiding permutation, which guarantees that $c^w_{u,v}=c^{w'}_{u,v'}$, and then using the fact that $\schub{v'}$ is a dominant monomial we can compute the Schubert structure coefficients by iterating the ``simplest'' case of Sottile's Pieri rule \cite{Sot96}. 
    \item Our second way is to show that the $c^w_{u,v}$ can be obtained by successively applying Schubert positive maps to $\schub{u}$, one of which is the divided difference operation $\partial_if=\frac{f-s_i\cdot f}{x_i-x_{i+1}}$ and the other is the Bergeron--Sottile map \cite{BS98,nst_c}
    $$\rope{i}f=f(x_1,\ldots,x_{i-1},0,x_i,x_{i+1},\ldots).$$
\end{enumerate}
Both ways have geometric interpretations. The first method corresponds to showing that $X^w_v$ is a left-translate of $X^{w'}_{v'}$ by the permutation matrix for $v'v^{-1}$. The second method corresponds to geometrically building $X^w_v$ in a manner similar to Schubert varieties via geometric push-pull operations, interspersed with inclusions of flag varieties into larger flag varieties associated to Bergeron--Sottile pattern maps \cite{BS98}, a technique that has been exploited by the authors in collaboration with Bergeron, Gagnon, and Nadeau  in a series of papers studying quasisymmetric coinvariants in algebraic geometry \cite{bgnst_1, bgnst_2, nst_c, NST_a}. 

Finally, Karp and Precup show that $X_{v_T}^{w_T}$ are smooth. We generalize this result by showing that the pairs $(v_T,w_T)$ satisfy a finer property we call \emph{very well-aligned}, and we show for very well-aligned pairs that $X_{v}^{w}$ is always a smooth variety.

\smallskip

\noindent \textbf{Outline:}
In Section~\ref{sec:background} we collect the combinatorial background. 
In Section~\ref{sec:well_aligned_and_bruhat} we introduce our central combinatorial object, well-aligned pairs, and then proceed to establish that the Bruhat interval $[v,w]$ coming from a well-aligned $(v,w)$ is equivalent by left translation to an interval $[v',w']$ where $v'$ is 132-avoiding. In \Cref{sec:WAschubert} we give two different nonnegative combinatorial expansions for $c^w_{u,v}$ for well-aligned $(v,w)$.
In Section~\ref{sec:richtab_to_well_aligned} we show that the pairs of permutations indexing the Richardson varieties of Karp and Precup are well-aligned. In Section~\ref{sec:class_expansions}
we reinterpret our combinatorial results geometrically. 
We conclude with some enumerative speculation in Section~\ref{sec:misc}.\\

\subsection*{Acknowledgements}
 We are very grateful to Martha Precup and Steven Karp for discussions about their work, for sharing notes as well as an early draft, and for directing us to relevant portions of their article. 
 We are also grateful to Allen Knutson for helpful discussions on descent cycling. 
 Finally, VT is particularly thankful to the organizers of the ``Combinatorics and Enumerative Geometry'' workshop held at the IAS in February 2025, as it provided ample food for thought.

\section{Combinatorial background}
\label{sec:background}
Throughout $[n]=\{1,\dots,n\}$ for $n$ a positive integer.
We refer the reader to standard texts \cite{Fulton, St99} for any undefined terminology.

\subsection{Young tableaux}

Recall that a \emph{partition} $\lambda$ is a weakly decreasing sequence of positive integers. 
We denote the size of $\lambda$, i.e. the sum of its entries by $|\lambda|$.
If $|\lambda|=n$, we denote this by $\lambda\vdash n$.
We represent $\lambda$ using its \emph{Young diagram} in English notation.
Given partitions $\mu\subseteq \lambda$, we define the \emph{skew shape} $\lambda/\mu$ as the set-theoretic difference of the Young diagrams of $\lambda$ and $\mu$.
A \emph{column strip} is a skew shape with no two boxes occupying the same row.

Given $\lambda\vdash n$, a \emph{standard Young tableau} (henceforth SYT) of shape $\lambda$ is a filling of the Young diagram of $\lambda$ bijectively with numbers drawn from $[n]$, so that the entries increase strictly from left to right along rows and from top to bottom along columns.
We let \emph{$\syt(\lambda)$} denote the set of standard Young tableaux of shape $\lambda$.

Given $T\in \syt(\lambda)$ we define two reading words associated with it.
The \emph{reading word} is obtained by reading the entries of the tableau row-wise bottom to top, with each row read left to right.  
The \emph{top-down reading word} is defined similarly except that the rows are read from top to bottom.

We now describe Sch\"utzenberger's evacuation operator \cite{Sch63}.
Given $T\in \syt(\lambda)$, define the evacuation tableau $\evac{T}\in \syt(\lambda)$ as follows. 
Delete the entry in the top left cell of $T$ and decrement the remaining entries by 1.
Then perform jeu-de-taquin slides to rectify the resulting Young tableau of skew shape thereby obtaining an SYT $T'$ on $n-1$ boxes with shape $\lambda\setminus\{c\}$, where $c$ is a corner box in $\lambda$.
We then place $n$ in $c$ and consider it ``frozen'' for the remainder of the procedure.
We repeat this with $T'$ and continue until all boxes are frozen.
The final tableau is $\evac{T}$.

\begin{eg}\label{ex:evacuation}
  Shown below are an SYT $T$ and its evacuation.
  \[
  \begin{ytableau}
        1 & 2 & 5 & 7 & 10\\
        3 & 8 & 11\\
        4 & 9\\
        6 & 12
\end{ytableau}\qquad
\begin{ytableau}
        1 & 4 & 7 & 9 & 12\\
        2 & 5 & 8\\
        3 & 10\\
        6 & 11
      \end{ytableau}
  \]
 The reading word of $T$ and the top-down reading word of $\evac{T}$ are as follows:
 $$\begin{tabular}{ccccccccccccc}
6 & 12 & 4 & 9 & 3 & 8& 11 & 1& 2 & 5& 7& 10\\
1 & 4 & 7 & 9& 12 & 2 & 5& 8& 3& 10& 6& 11
\end{tabular}$$
\end{eg}

\subsection{Permutations}
We denote the symmetric group comprising permutations of $[n]$ by $S_n$.
It is generated by the simple transpositions $s_i=(i,i+1)$ for $1\leq i\leq n-1$.
A \emph{reduced word} for $w\in S_n$ is a minimal length expression $w=s_{i_1}\cdots s_{i_{\ell}}$ as a product of simple transpositions.
An \emph{inversion} in $w\in S_n$ is an ordered pair $(i,j)$ where $i<j$ and $w(i)>w(j)$.
The \emph{length} of $w$, denoted by $\ell(w)$, is the number of inversions in $w$.
A \emph{descent} of $w$ is an index $i\in [n-1]$ such that $w(i)>w(i+1)$.
We denote the set of descents of $w$ by $\des{w}$.
We will typically write our permutations in one-line notation, e.g. $w = w(1)\cdots w(n)$.
We write $w_{\circ}$ for the longest element $n\, (n-1)\, \cdots 1$ in $S_n$.

\begin{defn}\label{de:bruhat_order}
  The \emph{Bruhat order $\leq_B$} on $S_n$ is obtained as the transitive closure of the cover relation $\lessdot_B$ defined by $u\lessdot_B v$ if and only if $ut=v$, where $t$ is a transposition and $\ell(v)-\ell(u)=1$.
\end{defn}

Since we need the tableau criterion \cite{BjBr05} for Bruhat order, we recall it here.
\begin{thm}[Tableau criterion]
  Given $u,v\in S_n$, we say that $u\leq_B v$ if and only if  $u_{i,k}\leq v_{i,k}$ for all $1\leq i\leq  k\leq n$, where $u_{i,k}$ is the $i$-th entry in the increasing rearrangement of $u(1),\dots,u(k)$ and similarly for $v_{i,k}$. In fact, it suffices to take $k\in \des{u}$ to check these inequalities, a statement known as  the \emph{improved tableau criterion} \cite[Theorem 2.6.3(ii)]{BjBr05}.
  \label{th:tableau_criterion} 
\end{thm}
Since we will need to apply the aforementioned criterion to establish incomparability in Bruhat order, the following definition will be useful.
\begin{defn}\label{def:witness}
If $u\nleq_B v$, then we call $p\in \des{u}$ a \emph{witness} if $u_{i,p}>v_{i,p}$ for some $1\leq i\leq p$.
\end{defn}

Given $u\leq_B v$, we let $[u,v]$ be the interval in Bruhat order containing all $w$ satisfying $u\leq_B w\leq_B v$.
We say that Bruhat intervals $[u,v]$ and $[p,q]$ are \emph{translation-equivalent} if left multiplication by $pu^{-1}$ restricts to an isomorphism of posets $[u,v]\xrightarrow{\sim}[p,q]$.
In particular we have the set-theoretic equality $pu^{-1}[u,v]=[p,q]$, or equivalently
\[
\{u^{-1}w\suchthat w\in [u,v]\}=\{p^{-1}r\suchthat r\in [p,q]\}.
\]
Some consequences of this definition will be used freely.
First, this isomorphism sends a cover $x\lessdot_B xt$ (with $t$ a transposition) to $pu^{-1}x\lessdot_B pu^{-1}xt$, and hence preserves the labelling of covers by transpositions.
Second, an isomorphism of intervals matches least and greatest elements, so $pu^{-1}u=p$ and $pu^{-1}v=q$; in particular $u^{-1}v=p^{-1}q$.

We record here a simple criterion that allows for producing certain translation-equivalent pairs. 
For the general statement for Coxeter groups we refer the reader to \cite[Proposition 1.9]{Del08} or \cite{duC00}.

\begin{prop}\label{pro:easy_bruhat_equivalent}
  Suppose $v\leq_B w$ where $v,w\in S_n$. 
Suppose $1\leq i\leq n-1$ is such that 
\begin{enumerate}[label=(\roman*)]
\item \label{item1} $v^{-1}(i)<v^{-1}(i+1)$,
\item \label{item2}$w^{-1}(i)<w^{-1}(i+1)$,
\item\label{item3} $s_iv\nleq_B w$.
\end{enumerate} 
Then left multiplication by $s_i=(s_iv)v^{-1}$ induces an isomorphism of posets $[v,w]\xrightarrow{\sim}[s_iv,s_iw]$; that is, $[v,w]$ and $[s_iv,s_iw]$ are translation-equivalent.
\end{prop}

The proposition is the special case $x=ww_{\circ}$ and $y=vw_{\circ}$ of \cite[Proposition 1.9]{Del08}.

\subsection{Rothe diagrams and dominant permutations}

Consider an $n\times n$ collection of boxes with rows labeled from $1$ to $n$ top to bottom and columns labeled from $1$ to $n$ left to right.
The \emph{Rothe diagram} of $w\in S_n$ is the collection of boxes
  \[
\rothe(w) = \{(w(j),i)\suchthat i<j,w(i)>w(j)\}.
  \]
Alternatively, $\rothe(w)$ is obtained by marking all boxes $b_i=(w(i),i)$ and crossing out all boxes below and to the right of $b_i$. 
The boxes that do not get crossed out form $\rothe(w)$.
The number of boxes in $\rothe(w)$ is equal to $\ell(w)$. 
See Figure~\ref{fig:rothe_43152} for two examples; the shaded boxes correspond to the cells of the Rothe diagram.

\begin{figure}[!h]
  \includegraphics[scale=0.6]{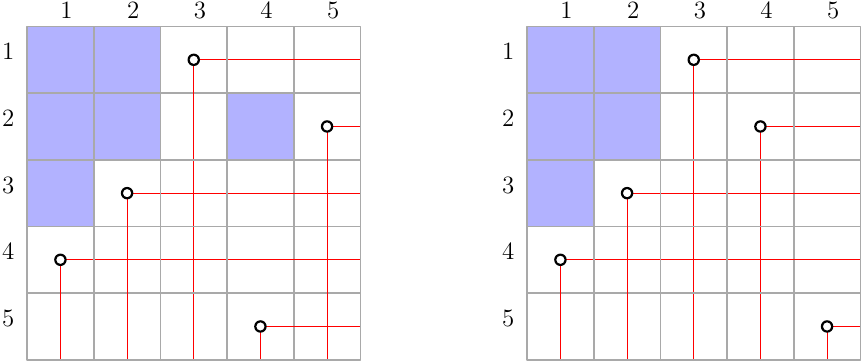}
  \caption{The Rothe diagrams for $43152$ (left) and dominant $43125$ (right).}
  \label{fig:rothe_43152}
\end{figure}
\begin{defn}
    A permutation is \emph{132-avoiding} or \emph{dominant} if there does not exist $i<j<k$ with $w(i)<w(k)<w(j)$.
\end{defn}
It is well known \cite{Knu68} that the set of dominant permutations in $S_n$ is enumerated by the Catalan number $\cat{n}\coloneqq \frac{1}{n+1}\binom{2n}{n}$; see also \cite[Chapter 2]{Man01}.
Note that $w$ is dominant if and only if $w^{-1}$ is dominant, since the pattern $132$ is an involution.
For $w\in S_n$ we write $\lcode{w}=(k_1,\ldots,k_n)$ for its \emph{Lehmer code}, where $k_i=\#\{j>i\suchthat w(j)<w(i)\}$.
As explained by Macdonald \cite[1.30]{Macdonald} (see also Manivel \cite[Chapter 2]{Man01}), a permutation $v$ is dominant if and only if its Rothe diagram is a Young diagram, in which case its row lengths are given by the Lehmer code $\lcode{v^{-1}}=(k_1,\ldots,k_p,0,\ldots,0)$ of $v^{-1}$.
See the right panel in Figure~\ref{fig:rothe_43152} for an example. In this case, we have $v^{-1}=34215$ and so its Lehmer code is $(2,2,1,0,0)$.


\subsection{Schubert polynomials}
We denote by $S_{\infty}=\langle s_1,s_2,\ldots\rangle$ the permutations of $\mathbb{N}$ with finite support, with $s_i=(i,i+1)$ the adjacent transpositions. We view $S_n\subset S_{\infty}$ as the subgroup generated by $s_1,\ldots,s_{n-1}$. We denote by $\mathbb{Z}[\xl_n]\coloneqq \mathbb{Z}[x_1,\ldots,x_n]$, and $\mathbb{Z}[\xl]\coloneqq \mathbb{Z}[x_1,x_2,\ldots]=\bigcup_{n=1}^{\infty}\mathbb{Z}[\xl_n]$. The group $S_{\infty}$ acts on $\mathbb{Z}[\xl]$ by having $s_i$ swap $x_i$ and $x_{i+1}$. The \emph{divided difference operations} are defined by $\partial_if=\frac{f-s_if}{x_i-x_{i+1}}$. These satisfy the \emph{nil-Coxeter relations} $\partial_i\partial_j=\partial_j\partial_i$ for $|i-j|\ge 2$, $\partial_i^2=0$, and the braid relations $\partial_i\partial_{i+1}\partial_i=\partial_{i+1}\partial_i\partial_{i+1}$. Because of this, we can define $\partial_w\coloneqq \partial_{i_1}\cdots \partial_{i_k}$ for any reduced word decomposition $w=s_{i_1}\cdots s_{i_k}$ of $w$.

The \emph{Schubert polynomials} $\{\schub{u}\suchthat u\in S_{\infty}\}\subset \mathbb{Z}[\xl]$ are the unique homogeneous polynomials with the property that $\ct\partial_w\schub{u}$ equals $1$ if $w=u$ and is $0$ otherwise. 
Alternately, they are characterized by $\schub{\idem}=1$ and
$$\partial_i\schub{u}=\begin{cases}\schub{us_i}&i\in \des{u}\\0&\text{otherwise.}\end{cases}$$
If $u\in S_n$, then for the longest element $w_{\circ}=n\,(n-1)\,\cdots\, 1$ we have $\schub{w_{\circ}}=x_1^{n-1}x_2^{n-2}\cdots x_{n-1}$, and we may explicitly compute $\schub{u}=\partial_{u^{-1}w_{\circ}}\schub{w_{\circ}}$.

We define the \emph{generalized Littlewood--Richardson coefficients} $c^w_{u,v}$ as the integers arising from the expansion
$$\schub{u}\schub{v}=\sum_w c^w_{u,v}\schub{w}.$$
Alternately, we may write $c^w_{u,v}=\ct\partial_w(\schub{u}\schub{v})$. It is known that if $u\not\le_B w$ or $v\not\le_B w$ then $c^w_{u,v}=0$.

The following fact will be essential for us.
\begin{fact}
\label{fact:translation_coeff}
    If $[v,w]$ is translation equivalent to $[v',w']$ then $c^w_{u,v}=c^{w'}_{u,v'}$. 
\end{fact}
There are several ways to see why this holds.
Translation-equivalence implies that the (skew) degree polynomials in the sense of Postnikov--Stanley \cite{PS09} are equal.
Indeed, this is an immediate consequence of the combinatorial description \cite[Section 12]{PS09} in terms of maximal chains in Bruhat order; the fact that translation-equivalence does not alter the transpositions labeling cover relations implies the claimed equality.
Fact~\ref{fact:translation_coeff} then follows from the fact that generalized Littlewood--Richardson coefficients are precisely the coefficients that arise when one decomposes skew degree polynomials in terms of ordinary degree polynomials; see \cite[Corollary 6.9]{PS09}.

We give an additional geometric proof in Section~\ref{sec:class_expansions}: translation-equivalence realizes $X^{w'}_{v'}$ as a left translate of $X^w_v$ by a permutation matrix (Proposition~\ref{pro:richardson_translate}), and since left translation acts trivially on homology the two Richardson varieties have the same Schubert expansion. See the discussion following Corollary~\ref{co:well_aligned_translate}. 
This proof is independent of Fact~\ref{fact:translation_coeff}, so no circularity arises from its forward reference.

\section{Well-aligned pairs}
\label{sec:well_aligned_and_bruhat}

The following map that inserts $1$ into a permutation shall be relevant to us.
Given a permutation $w\in S_{n-1}$ and a positive integer $j\in \{1,\dots,n\}$, let $\ins_{j}(w)\in S_n$ be the permutation obtained by inserting a $1$ in position $j$ in $w$ and then incrementing the previously existing numbers by $1$.
Similarly, given $w\in S_n$ we define  $\delta(w)\in S_{n-1}$  to be the permutation obtained by deleting $1$ from $w$ and then decrementing the remaining numbers by 1. Alternatively
$$(\ins_iw)(j)=\begin{cases}w(j)+1&j<i\\
1&j=i\\
w(j-1)+1&j>i\end{cases}\qquad\qquad\delta(w)(j)=\begin{cases}w(j)-1&j<w^{-1}(1)\\
w(j+1)-1&j\ge w^{-1}(1).\end{cases}$$
Writing permutations in one line notation, we have for instance that $\ins_1(25143)=136254$, $\ins_2(25143)=316254$, and $\delta(25143)=1432$.

\begin{defn}
\label{def:good}
    We say a pair of permutations $(v,w)\in S_n\times S_n$ is \emph{aligned} if the following hold:
    \begin{enumerate}
        \item $v^{-1}(1)\leq w^{-1}(1)$,
        \item all indices $v^{-1}(1)\leq i\leq w^{-1}(1)-1$ are ascents in $v$, i.e. satisfy $v(i)<v(i+1)$.
    \end{enumerate}
    We call $(v,w)\in S_n\times S_n$ \emph{well-aligned} if it is aligned and $(\delta(v),\delta(w))$ is well-aligned, with the convention that the unique pair $(\mathrm{id},\mathrm{id})\in S_0\times S_0$ is well-aligned as a base case.
    We let $\wellaligned{n}\subset S_n\times S_n$ denote the set of well-aligned pairs.
\end{defn}

\begin{eg}
\label{ex:from_martha_notes}
  We have $(\textcolor{red}{1}5726348, 75\textcolor{red}{1}82364)$ is aligned because $1<5<7$, and is in fact well-aligned, as successive applications of $\delta$ yield \begin{align*}(\textcolor{red}{1}5726348, 75\textcolor{red}{1}82364)\mapsto (46\textcolor{red}{1}5237,647\textcolor{red}{1}253)&\mapsto (354\textcolor{red}{1}26,536\textcolor{red}{1}42)\mapsto (243\textcolor{red}{1}5,4253\textcolor{red}{1})\\&\mapsto (\textcolor{red}{1}324,3\textcolor{red}{1}42)\mapsto (2\textcolor{red}{1}3,23\textcolor{red}{1})\mapsto (\textcolor{red}{1}2,\textcolor{red}{1}2)\mapsto (\textcolor{red}{1},\textcolor{red}{1}).\end{align*}
\end{eg}

\begin{eg}
\label{ex:w_wc}
  A class of well-aligned pairs arises organically in a geometric context in the authors' previous work \cite{nst_c} joint with Nadeau.
  Let $\cox=s_{n-1}\cdots s_2s_1\in S_n$ be a standard Coxeter element. 
  Then $(v,v\cox)$ is well-aligned for all $v\in S_n$ satisfying $v(n)=n$.
\end{eg}
The main result of this section will be \Cref{pr:to_dominant}, that for a well-aligned pair $(v,w)$, there is a distinguished dominant permutation $v^\uparrow$ with $[v,w]$ translation-equivalent to $[v^{\uparrow},v^{\uparrow}v^{-1}w]$.

\subsection{Translation-equivalent Bruhat intervals and well-aligned pairs}

Our next lemma relates well-aligned pairs to Bruhat order.
\begin{lem}\label{le:well_aligned_implies_bruhat}
For $v,w\in S_n$, 
  if $(v,w)\in \wellaligned{n}$ then $v\leq_B w$.
 \end{lem}
\begin{proof}
  We establish the claim by induction on $n$. 
  Let $i=v^{-1}(1)$ and $j=w^{-1}(1)$, so that $v=\ins_j(\delta(v))s_{j-1}\cdots s_i$ and $w=\ins_j(\delta(w))$.
  We have
  $$
  v\le_B vs_i\le_B \cdots \le_B vs_i\cdots s_{j-1}=\ins_j(\delta(v))\le_B \ins_j(\delta(w)),
  $$
  where the first string of inequalities is because $(vs_i\cdots s_{k-1})(k)=v(i)=1<(vs_i\cdots s_{k-1})(k+1)$ and the last inequality is by the inductive hypothesis $\delta(v)\le_B \delta(w)$ and the fact that $\ins_j$ preserves the Bruhat order by the tableau criterion.
\end{proof}


The next result provides a converse when $v$ is dominant. 
\begin{lem}
     Let $v\in S_n$ be dominant. Then $v \leq_B w$, if and only if $(v,w)\in \wellaligned{n}$.
\end{lem}
\begin{proof}
The previous lemma shows the reverse implication, so it remains to establish the forward implication.

    Let $j=v^{-1}(1)$. Since $v$ is dominant we must have that $1=v(j)<v(j+1)<\cdots <v(n)$.
    It is well known (see for instance \cite[\S 2]{MJVKIM16}) that for $v$ dominant,
    \begin{equation}
    \label{eqn:rotheBruhat}v\leq_B w \Longleftrightarrow \rothe(v)\subseteq \rothe(w).\end{equation}
    In particular $w^{-1}(1)\geq j$, and so $(v,w)$ is aligned.

    To finish the proof we need to show that $(\delta(v),\delta(w))\in \wellaligned{n-1}$.
    By induction it suffices to show that $\delta(v)\leq_B \delta(w)$.
    Note that $\delta(v)$ is dominant.
    Furthermore, since $\rothe(v)\subseteq \rothe(w)$, we have that $\rothe(\delta(v))\subseteq \rothe(\delta(w))$.
    This implies that $\delta(v)\leq_B \delta(w)$ as desired. See Figure~\ref{fig:rothe_inclusion_dominant} for an illustration of the inclusion of the Rothe diagrams.
\end{proof}

\begin{figure}[!h]
  \includegraphics[scale=0.6]{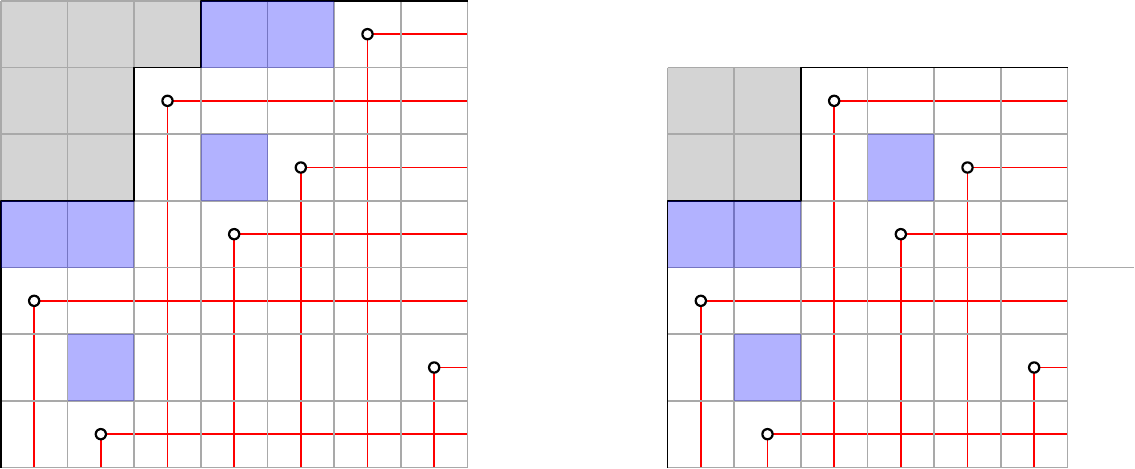}
  \caption{
The gray shaded cells correspond to the Rothe diagram of the dominant permutations $v=4521367$ and $\delta(v)=341256$ respectively, while the union of the gray and blue shaded cells correspond to the Rothe diagrams of $w=5724316$ and $\delta(w)=461325$ respectively.}  
  \label{fig:rothe_inclusion_dominant}
\end{figure}

Define
\[
\wellaligned{n}^{132}\coloneqq \{(v,w)\in \wellaligned{n}\suchthat v \text{ is  dominant}\}.
\]
\begin{rem}\label{re:double_factorial}
  Pairs $(v,w)\in S_n\times S_n$ with $v\leq_B w$ where $v$ is dominant are in bijection with matchings on $2n$ points \cite[\S 2]{MJVKIM16}.
  Thus $|\wellaligned{n}^{132}|=(2n-1)!!$.
\end{rem}

We emphasize a special instance of 132-patterns.
Given $w\in S_n$ we call $1\leq i\leq n-1$ \emph{critical} if $w$ possesses a subsequence of the form $\cdots  i \cdots j \cdots (i+1)\cdots $ where $j>i+1$. 
We let the set of critical values of $w$ be denoted by $C(w)$.
It is straightforward to check that 
\[C(w)
\text{ is empty } \Longleftrightarrow w\text{ is dominant}. 
\]

\begin{lem}\label{le:condition_2}
Suppose $(v,w)\in \wellaligned{n}$.
    Then we have $w^{-1}(i)<w^{-1}(i+1)$ for all $i\in C(v)$.
\end{lem}
\begin{proof}
  We first consider the case $i=1$.
  Since $(v,w)$ is aligned, the criticality of $1$ implies that we must have $w^{-1}(1)<v^{-1}(2)$.
  We claim this implies that $w^{-1}(1)<w^{-1}(2)$.
  Indeed, if $w^{-1}(1)>w^{-1}(2)$ then $(\delta(v),\delta(w))$  cannot be aligned which contradicts the well-alignedness of $(v,w)$.
  
  Now suppose $i>1$. 
  Note that if $i\in C(v)$, then $i-1\in C(\delta(v))$.
  Since $(\delta(v),\delta(w))$ is well-aligned, our inductive hypothesis implies that $(\delta(w))^{-1}(i-1)<(\delta(w))^{-1}(i)$.
  But then we necessarily have $w^{-1}(i)<w^{-1}(i+1)$.
\end{proof}

\begin{eg}
For $(v,w)=(15726348, 75182364)\in \wellaligned{8}$, we have $C(v)=\{1,2,5\}$. 
By inspection for each $i\in C(v)$ we have $i$ lying to the left of $i+1$ in $w$.
\end{eg}

\begin{lem}\label{le:well_aligned_preserved}
    Suppose $(v,w)\in \wellaligned{n}$.
    Then we have that $(s_iv,s_iw) \in \wellaligned{n}$ for all $i\in C(v)$.
\end{lem}
\begin{proof}
  It is straightforward to see that 
 $(s_iv,s_iw)$ is aligned (when $i=1$ this uses that $(\delta(v),\delta(w))$ is aligned, as in the proof of Lemma~\ref{le:condition_2}), so it remains to show $(\delta(s_iv),\delta(s_iw))\in \wellaligned{n-1}$. 
 We show this by induction on $i$. 
 For $i=1$, begin by observing that $\delta(s_1v)$ and $\delta(s_1w)$ are obtained by deleting the $2$ from the one line notation of $v$ and $w$ respectively and then decrementing all entries greater than two by one.
  This implies that $(\delta(s_1v),\delta(s_1w))$ is aligned.
  Furthermore we have 
  \[
  (\delta\delta(s_1v),\delta\delta(s_1w))=(\delta\delta(v),\delta\delta (w)),
  \] 
  and the latter is well-aligned because $(v,w)$ is. 
  Thus $(\delta\delta(s_1v),\delta\delta(s_1w))\in \wellaligned{n-2}$, which in turn implies $(\delta(s_1v),\delta(s_1w))\in \wellaligned{n-1}$ and  concludes the base step of our induction.
  
  Now suppose $i>1$.
  Note that $(\delta(s_iv),\delta(s_iw))=(s_{i-1}\delta(v),s_{i-1}\delta(w))$. Furthermore, observe that $(\delta(v),\delta(w))\in \wellaligned{n-1}$ with $i-1\in C(\delta(v))$.
  The inductive hypothesis thus implies that $(s_{i-1}\delta(v),s_{i-1}\delta(w))\in \wellaligned{n-1}$, and therefore that $(\delta(s_iv),\delta(s_iw))\in \wellaligned{n-1}$.
  This concludes the induction step.
\end{proof}

\begin{lem}
\label{le:translation_equivalent}
  Suppose $(v,w)\in \wellaligned{n}$.
  For all $i\in C(v)$, we have that $[s_iv,s_iw]$ is translation-equivalent to $[v,w]$ .
\end{lem}
\begin{proof}
  We induct on $i$, verifying conditions in  Proposition~\ref{pro:easy_bruhat_equivalent}\ref{item1}--\ref{item3}.
  Condition~\ref{item1} holds because $i\in C(v)$.
  Condition~\ref{item2} is the content of Lemma~\ref{le:condition_2}. 
  It remains to check Condition~\ref{item3} that $s_iv\nleq w$.

  Suppose $i=1$. 
  Let $j=v^{-1}(2)$ and $k=w^{-1}(1)$, and note that $1\in C(v)$ implies that $j>k$.
  Therefore $\min\{(s_1v)(1),\ldots, (s_1v)(k)\}=2>1=\min\{w(1),\ldots,w(k)\}$ and we conclude $s_1v \nleq_B w$ by the tableau criterion.

  Now suppose $i>1$.
  Then $i-1\in C(\delta(v))$ so our inductive hypothesis yields $
  s_{i-1}\delta(v)\nleq_B \delta(w)$.
  Let $p\in \des{s_{i-1}\delta(v)}$ be a witness for this incomparability; see Definition~\ref{def:witness}.
  Suppose that 
  \[
  s_iv=\ins_a(s_{i-1}\delta(v))\quad\text{and}\quad w=\ins_b(\delta(w)).
  \]
  Note that $a\leq b$ since $(v,w)$ is aligned. If furthermore $a\leq p$ then $b\leq p+1$ by alignedness of $(v,w)$, so $p+1\in \des{s_iv}$ is a witness for $s_iv\nleq_B w$.
  On the other hand if $a> p$, then $b>p$ by alignedness of $(v,w)$ so $p\in \des{s_iv}$ witnesses the incomparability $s_iv\nleq_B w$.
\end{proof}

We illustrate various aspects of the preceding lemmas with the following example.
\begin{eg}
For $(v,w)=(15726348, 75182364)\in \wellaligned{8}$, we have $C(v)=\{1,2,5\}$. 
It is easily checked that $s_1v\nleq_B w$.

Consider $i=2$. 
We will show how a witness for $s_1\delta(v)\nleq_B \delta(w)$ produces one for $s_2v\nleq_B w$.
We have $\delta(v)=4615237$ and $\delta(w)=6471253$.
Since the first four entries of $s_1\delta(v)$ in increasing order are $2,4,5,6$ while those of $\delta(w)$ are $1,4,6,7$, we have that $p=4$ witnesses $s_1\delta(v)\nleq_B \delta(w)$.
Now note that 
\[
s_2v=15736248=\ins_1(s_1\delta(v))\quad\text{and}\quad w=\ins_3(\delta(w)).
\] 
Therefore, $p+1=5$ witnesses $s_2v\nleq_B w$.

Finally, consider $i=5$.
Since the first two entries of $s_4\delta(v)=5614237$ in increasing order are $5,6$ while those of $\delta(w)$ are $4,6$, we have that $p=2$ witnesses $s_4\delta(v)\nleq_B \delta(w)$.
Now note that 
\[
s_5v=16725348=\ins_1(s_4\delta(v))\quad\text{and}\quad w=\ins_3(\delta(w)).
\] 
Therefore, $p+1=3$ witnesses $s_5v\nleq_B w$.
\end{eg}

\subsection{Labeled plane binary trees and dominant permutations}

Recall that a \emph{plane binary tree} $T$ is a rooted tree where every internal (i.e. non-leaf) node has a left and right subtree. We denote the set of such trees with $n$ internal nodes by $\mc{T}_n$. A decreasing labeling $\mathcal{L}$ of $T\in \mc{T}_n$ is a bijection of the internal nodes with $[n]$ such that whenever $w$ is a descendent of $v$, the label of $w$ is less than the label of $v$. We now recall a folklore bijection \cite[pp. 23--24]{St97}
\begin{align*}
S_n&\to \{(T,\mathcal{L})\suchthat T\in \mc{T}_n\text{ and }\mathcal{L}\text{ is a decreasing labeling of $T$}\}\\
w&\mapsto (\psi(w),\mathcal{L}(w))
\end{align*}
 though we follow the conventions in \cite[\S 2.2]{LoRo98}.

The bijection is obtained by recursively applying the following procedure, starting from the one-line notation $w(1)\cdots w(n)$ of $w$: for a word of distinct numbers $z$ we write $z=z^{(1)}mz^{(2)}$ with $m=\max(z)$ and associate to it the tree whose root is labeled $m$ and whose left and right subtrees are given recursively by applying this procedure to the words $z^{(1)}$ and $z^{(2)}$ respectively.

Given $T\in \mc{T}_n$, consider the fiber 
\[
Z_T=\{w\in S_n\suchthat \psi(w)=T\}.
\]
It is well known that $Z_T$ has a unique maximal element $w^{\uparrow}$ under left weak order, characterized by the fact that it is the unique dominant permutation in $Z_T$ (see \cite[Theorem 2.5]{LoRo02} and also \cite[\S 1.2]{AS06}). 
In particular for dominant permutations we have $w=w^{\uparrow}$.
We can describe $w^{\uparrow}$ explicitly as the unique permutation in $Z_T$ such that $\psi(w^{\uparrow})$ has the property that for any node $v$, the  smallest label in the left subtree of $v$ is greater than the largest label in the right subtree of $v$.
Figure~\ref{fig:T_and_Z_T} shows an unlabeled plane binary tree $T$ on the left and the set $Z_T$ on the right, with the dominant permutation highlighted.

\begin{figure}[!ht]
\includegraphics[scale=0.8]{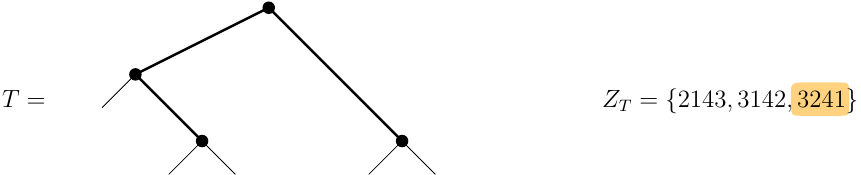}
\caption{A plane binary tree $T$ and the set $Z_T$.}
\label{fig:T_and_Z_T} 
\end{figure}

Note that $i\in C(w)$ is equivalent to stating that the nodes labeled $i$ and $i+1$ are incomparable in $T$ with the node labeled $i$ appearing before that labeled $i+1$ in the inorder traversal.
Because all linear extensions of a poset can be obtained by swapping labels of adjacent incomparable nodes, by repeatedly applying simple transpositions corresponding to the entries in the critical set we may transform $w$ to $w^{\uparrow}$ without altering the underlying tree.
Figure~\ref{fig:three_extensions} illustrates this starting with $w=2143$ and applying swaps corresponding to the critical values.

\begin{figure}[!ht]
\includegraphics[scale=0.8]{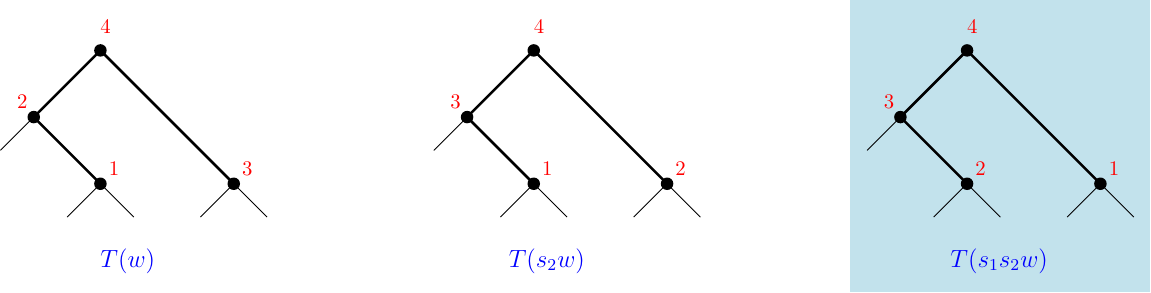}
\caption{Three decreasing trees with the leftmost tree associated to $w=2143$ and the rightmost corresponding to the dominant permutation $w^{\uparrow}=s_1s_2w=3241$.}
\label{fig:three_extensions}
\end{figure}



We are now ready to prove the main result of this section.
\begin{prop}\label{pr:to_dominant}
    Suppose $(v,w)\in\wellaligned{n}$.
    Then the following hold.
    \begin{enumerate}
    \item $(v^{\uparrow},v^{\uparrow}v^{-1}w)\in\wellaligned{n}^{132}$.
    \item The intervals $[v,w]$ and $[v^{\uparrow},v^{\uparrow}v^{-1}w]$ are translation-equivalent.
    \end{enumerate}
\end{prop}
\begin{proof}
  If $v$ is dominant then there is nothing to show as $v=v^{\uparrow}$.
  Thus we may assume that $v$ is not dominant, i.e. $C(v)\neq \emptyset$.
  Pick any $i\in C(v)$.
  Then $(s_iv,s_iw)\in \wellaligned{n}$ by Lemma~\ref{le:well_aligned_preserved}.
  Furthermore we know that $\ell(s_iv)=\ell(v)+1$ and $\ell(s_iw)=\ell(w)+1$.
  Repeating this procedure eventually produces a $(v',w')\in \wellaligned{n}^{132}$ so that $[v',w']$ is translation-equivalent to $[v,w]$.
  Since $v$ and $v'$ are in the same fiber of $\psi$, we have $v'=v^{\uparrow}$.
  Both claims now follow.
\end{proof}

\begin{eg}
For $(v,w)=(15726348, 75182364)\in \wellaligned{8}$, we have $v^{\uparrow}=56734128$. Proposition~\ref{pr:to_dominant} guarantees that $[v,w]$ is translation-equivalent to $[v^{\uparrow},v^{\uparrow}v^{-1}w]=[56734128, 76583142]$.
\end{eg}

\begin{rem}
  We revisit the well-aligned pairs from Example~\ref{ex:w_wc}.
  In this case, Proposition~\ref{pr:to_dominant} guarantees that the Bruhat interval $[v,v\cox]$ is translation-equivalent to $[v^{\uparrow},v^{\uparrow}\cox]$ where $v^{\uparrow}\in S_n$ is dominant and satisfies $v^{\uparrow}(n)=n$ as $v$ itself does. 
  In particular there are $\cat{n-1}$ many such intervals, and it is a fact that they are all pairwise inequivalent with respect to translation. 
  Translating these special intervals by $(v^{\uparrow})^{-1}$ produces noncrossing partitions, and this connection is the first step toward understanding quasisymmetry in the context of $\flag{n}$ \cite{bgnst_2,nst_c}.
\end{rem}

\section{Well-aligned Schubert structure coefficients}
\label{sec:WAschubert}
Recall that we want to determine the structure coefficients $c^w_{u,v}$ where $(v,w)\in \wellaligned{n}$.
\subsection{Coefficients via translation-equivalence}
We recall Sottile's Pieri rule \cite{Sot96} in a special case as it will be the only combinatorial gadget that we need.
We momentarily work in the infinite symmetric group $S_{\infty}$.
For $1\leq a<b$, let $t_{ab}\in S_{\infty}$ be the transposition swapping $a$ and $b$.

\begin{defn}\label{de:k_bruhat}
Given a nonnegative integer $k$, we say that a permutation $u$ is covered by $ut_{ab}$ in \emph{$k$-Bruhat order} if $a\leq k<b$ and $\ell(u\,t_{ab})=\ell(u)+1$.
The $k$-Bruhat order $u\leq_B^k w$ is obtained as the transitive closure of these covers.
\end{defn}

The saturated chains that matter to us are 
\[
  u\lessdot_B^k u\,t_{a_1b_1}\lessdot_B^k u \,t_{a_1b_1}\,t_{a_2b_2}\lessdot_B^k \cdots \lessdot_B^k u \,t_{a_1b_1}\cdots t_{a_kb_k}=w
\]
where the $t_{a_ib_i}$ are transpositions with $a_1$ through $a_k$ distinct and each $a_i\leq k<b_i$.
We write $u\stackrel{k}{\to} w$ if there \emph{exists} a chain as above from $u$ to $w$.

\begin{thm}[{\cite[Theorem 1]{Sot96}}]
\label{th:sottile_pieri}
Fix $k$ a nonnegative integer and let $u\in S_{\infty}$.
Then we have
\begin{align*}
x_1\cdots x_k\schub{u}=\sum
 \schub{w},
\end{align*}
where the sum ranges over all $w$ with $u\stackrel{k}{\to}w$. 
In particular, this expansion is multiplicity-free.
\end{thm}
Recall that for a dominant permutation $v$ the Schubert polynomial is the monomial $\schub{v}=x_1^{c_1}x_2^{c_2}\cdots$ with exponents the Lehmer code $\lcode{v}=(c_1,c_2,\dots)$ of $v$; see \cite[Chapter 2]{Man01}. Since $\lcode{v}$ and $\lcode{v^{-1}}$ are conjugate partitions for $v$ dominant, writing the Lehmer code of $v^{-1}$ as $(k_1,\dots,k_p,0,0,\dots)$ we equivalently have
$$\schub{v}=\prod_{1\leq i\leq p}x_1\cdots x_{k_i},$$
a dominant monomial.
Iterating the Pieri rule in Theorem~\ref{th:sottile_pieri} tells us how to compute the product of an arbitrary Schubert polynomial with such a dominant monomial in the Schubert basis.
Suppose we have a sequence $\mathbf{k}=(k_1\geq \cdots \geq k_p)$ of nonnegative integers and a permutation $u\in S_{\infty}$.
We write $u\stackrel{\mathbf{k}}{\to} w$ for a sequence of permutations $u=u_0,u_1,\dots,u_p=w$ such that $u_{i-1}\stackrel{k_i}{\to} u_i$ for $1\leq i\leq p$.
Note that a given $w$ may arise from several such sequences.


We thus have the following corollary.
\begin{cor}
\label{co:iterated_pieri}
Let $v\in S_{\infty}$ be dominant with Lehmer code of $v^{-1}$ given by $\mathbf{k}=(k_1,\ldots,k_p,0,0,\ldots)$. 
Then we have
\begin{align*}
\schub{v}\schub{u}=\sum_{u\stackrel{\mathbf{k}}{\to} w}
 \schub{w},
\end{align*}
where the sum ranges over all $u\stackrel{\mathbf{k}}{\to} w$.
In particular the Schubert structure coefficients $c^w_{u,v}$ can be computed combinatorially. Unlike the expansion in Theorem~\ref{th:sottile_pieri} this one need not be multiplicity-free.
\end{cor}

We are now in a position to state our main result.
\begin{thm}\label{th:well_aligned_to_dominant}
Let $(v,w)\in \wellaligned{n}$.
Then
$$c^w_{u,v}=c^{v^{\uparrow}v^{-1}w}_{u,v^{\uparrow}},$$
and these coefficients can be computed combinatorially.
\end{thm}
\begin{proof}
    By Proposition~\ref{pr:to_dominant}, the intervals $[v,w]$ and $[v^{\uparrow},v^{\uparrow}v^{-1}w]$ are translation-equivalent. The equality $c^w_{u,v}=c^{v^{\uparrow}v^{-1}w}_{u,v^{\uparrow}}$ therefore follows from Fact~\ref{fact:translation_coeff}.
    Since $v^{\uparrow}$ is dominant, these coefficients can in turn be computed combinatorially by Corollary~\ref{co:iterated_pieri}.
\end{proof}

\subsection{Coefficients via Bergeron--Sottile maps}
We now proceed to study the well-aligned Schubert structure coefficients $c^w_{u,v}$ in a different way, using the Bergeron--Sottile maps. 
Say that a map $f:\mathbb{Z}[\xl]\to \mathbb{Z}[\xl]$ is \emph{Schubert positive} if $$f(\schub{u})=\sum a^v_u\schub{v}\text{ with }a^v_u\ge 0$$
where $u,v\in S_{\infty}$.
We will say that $f$ is \emph{combinatorially Schubert positive} if we can combinatorially determine $a^v_u$ to be nonnegative. The functional
$\Phi^w_v:\mathbb{Z}[\xl]\to \mathbb{Z}$
given by
$$\Phi^w_v(f)=\ct \partial_w(f\schub{v})$$
is Schubert positive because $\Phi^w_v(\schub{u})=c^w_{u,v}$. To show the combinatorial nonnegativity of $c^w_{u,v}$ for well-aligned pairs $(v,w)\in \wellaligned{n}$ is tantamount to showing that $\Phi^w_v$ is in fact combinatorially Schubert positive.

We will show the combinatorial Schubert positivity of $\Phi^w_v$ by showing it is a composite of combinatorially Schubert positive operations. The operations $\ct,\partial_i$ are manifestly combinatorially Schubert positive operations; we will need one further combinatorially Schubert positive operation.
\begin{defn}
  For $f\in \ZZ[\xl]$ and a positive integer $i$, we define the \emph{$i$th Bergeron--Sottile operator} by 
  \[
    \rope{i}f=f(x_1,\ldots,x_{i-1},0,x_i,x_{i+1},\ldots)
  \] 
  That is, $\rope{i}$ sets $x_i=0$ and relabels $x_j\mapsto x_{j-1}$ for $j>i$; so $\rope{i}$ depends only on $i$ (and not on any ambient number of variables) as an operator on $\ZZ[\xl]$.
\end{defn}

As shown in Bergeron--Sottile \cite{BS98}, the result of applying $\rope{i}$ to Schubert polynomials can be computed combinatorially using Sottile's Pieri rule -- in particular, $\rope{i}$ is a combinatorially nonnegative operation. 
We verify this here quickly.
\begin{fact}
We have $\rope{1}\schub{w}=\schub{\delta(w)}$ if $w(1)=1$, and $0$ otherwise. Furthermore
$$\rope{i}f=\rope{1}\partial_1\cdots \partial_{i-1} x_1\cdots x_{i-1}f.$$
\end{fact}
\begin{proof}
The statement for $\rope{1}$ is well-known -- one way to see it is that  because $\partial_i\rope{1}=\rope{1}\partial_{i+1}$ for all $i$, we have $\ct\partial_{w'}\rope{1}\schub{w}=\ct\partial_{\ins_1(w')}\schub{w}$, which equals $1$ if $\ins_1(w')=w$ and $0$ otherwise. In other words, we have that $\ct\partial_{w'}\rope{1}\schub{w}$ equals $1$ if and only if $w(1)=1$ and $\delta(w)=w'.$
 For the more general statement, because $\partial_j$ commutes with polynomials symmetric in $x_j,x_{j+1}$ we may rewrite
$$\rope{1}\partial_1\cdots \partial_{i-1}x_1\cdots x_{i-1}f=\rope{1}\partial_1x_1\partial_2x_2\cdots \partial_{i-1}x_{i-1}f$$
and then repeatedly apply the fact $\rope{j}\partial_jx_j=\rope{j+1}$.
\end{proof}

\begin{lem}
\label{lem:Phisimplify}
  Consider $(v,w)\in \wellaligned{n}$. Suppose $i=v^{-1}(1)$ and $j=w^{-1}(1)$. 
  Then
  \[
    \Phi^w_v=\Phi^{\delta(w)}_{\delta(v)}\rope{i}\partial_{i}\partial_{i+1} \cdots \partial_{j-1}.
  \]
\end{lem}
\begin{proof}
  If $i=j$, then $\Phi^w_v=\Phi^{\delta(w)}_{\delta(v)}\rope{i}$ is the statement of \cite[Proposition 8.1]{nst_c}.
  Otherwise because $(v,w)$ is well-aligned we have $j-1\in \des{w}$ and $j-1\not\in \des{v}$ so $$\Phi^w_vf=\ct\partial_w\schub{v}f=\ct\partial_{ws_{j-1}}\partial_{j-1}\schub{v}f=\ct\partial_{ws_{j-1}}\schub{v}(\partial_{j-1} f)=\Phi^{ws_{j-1}}_v\partial_{j-1}f,$$ where the second last equality follows because $\schub{v}$ is symmetric in the variables $x_{j-1},x_{j}$ (this idea is a special case of Knutson's descent-cycling \cite{Knu01}).
  Since $(v,ws_{j-1})\in \wellaligned{n}$, we can apply this repeatedly until the positions of $1$ align and then apply the case $i=j$.
\end{proof}

Note that since $(\delta(v),\delta(w))\in \wellaligned{n-1}$ if $(v,w)\in \wellaligned{n}$, iterating the previous lemma, we obtain a formula for $\ct\partial_w(\schub{v}\schub{u})$ as a composite of $\rope{i}$s and $\partial_j$s applied to $\schub{u}$. We make this explicit in the case where $(v,w)\in \wellaligned{n}^{132}$.
In this case, this composite can be read off from the skew Rothe diagram $\rothe(w)\setminus \rothe(v)$. 
\begin{prop}
  Let $(v,w)\in \wellaligned{n}^{132}$ and let $D\coloneqq \rothe(w)\setminus \rothe(v)$.
  For $1\leq i\leq n$, let $r_i$ be the number of boxes in the $i$th row of $D$, considered from top to bottom.
  Suppose further that $(c_1,\dots,c_n)=\lcode{v^{-1}}$.
  Then we have
  \[
    \Phi^w_v=\ev_0 \rope{c_n+1}\partial_{c_n+1}\cdots \partial_{c_n+r_n}
    \rope{c_{n-1}+1}\partial_{c_{n-1}+1}\cdots \partial_{c_{n-1}+r_{n-1}}\cdots
    \rope{c_1+1}\partial_{c_1+1}\cdots \partial_{c_1+r_1}.
  \]
\end{prop}
\begin{proof}
      Iterating Lemma~\ref{lem:Phisimplify} expresses $\Phi^w_v$ as the claimed composite: applying $\delta$ to each of $v$ and $w$ corresponds to striking out the top row of the bounding $n\times n$ box and then the column containing the $1$ in $w$. 
\end{proof}

\begin{eg}
  Referring back to the well-aligned pair
  $(v,w)=(4521367,5724316)$ depicted on the left in Figure~\ref{fig:rothe_inclusion_dominant}, we have 
  \begin{align*}
  (r_1,\dots,r_7)=(2,0,1,2,0,1,0)\quad \text{and} \quad 
  (c_1,\dots,c_7)=(3,2,2,0,0,0,0)=\lcode{4351267}.
  \end{align*}
  Thus we have
  \[
    \Phi^w_v=\ev_0 \rope{1}\rope{1}\partial_{1}\rope{1}\rope{1}\partial_{1}\partial_{2}\rope{3}\partial_{3}\rope{3}\rope{4}\partial_{4}\partial_{5}.
  \]
  In particular, the boundary of the grey region corresponding to $v$ determines the subscripts of the $\rope{i}$s and the skew Rothe diagram $D$ (given by the blue shaded cells) determines the subscripts of the $\partial_j$s.
\end{eg}

\section{Richardson tableaux and well-aligned pairs}
\label{sec:richtab_to_well_aligned}

The main subfamily of well-aligned pairs in this article comes from a subfamily of SYTs comprising Richardson tableaux.
Given a partition $\lambda$, we let \emph{$\crop(\lambda)$} denote the partition obtained by omitting the first row of $\lambda$.
For $T\in \syt(\lambda)$, we let \emph{$\crop(T)$} denote the SYT of shape $\crop(\lambda)$ obtained from the subtableau of $T$ determined from rows 2 and below, where we naturally standardize said filling so that the entries are precisely 
$\{1,\dots, |\crop(\lambda)|\}$. 

\begin{defn}\label{def:richtab}
We say that an SYT $T$ is a \emph{Richardson tableau} if the following conditions hold:
  \begin{enumerate}[label=(\arabic*)]
    \item \label{it1:richtab_condition_1}for every entry $j$ in the second row $j-1$ is in the first row, and
    \item \label{it2:richtab_condition_2} $\crop(T)$ is a Richardson tableau,
  \end{enumerate}
with the convention that the empty tableau is a Richardson tableau as a base case.
Given a partition $\lambda$ we let $\richtab(\lambda)$ denote the set of Richardson tableaux of shape $\lambda$.
\end{defn}
The preceding definition is not the original definition of Richardson tableau by Karp and Precup \cite[Definition 1.3]{KaPr25}; that the aforementioned definition is an alternative characterization is shown by the same authors  \cite[Corollary 4.18]{KaPr25}.

\begin{eg}\label{ex:cropping}
  Shown below are repeated croppings of the tableau on the left. Each tableau satisfies condition~\ref{it1:richtab_condition_1} in Definition~\ref{def:richtab}.
  In particular the leftmost tableau is an element of $\richtab((5,3,2,2))$.
  \begin{align*}
      \begin{ytableau}
        1 & 2 & 5 & 7 & 10\\
        3 & 8 & 11\\
        4 & 9\\
        6 & 12
      \end{ytableau}\qquad
      \begin{ytableau}
        1 & 4 & 6\\
        2 & 5\\
        3 & 7
      \end{ytableau}\qquad
      \begin{ytableau}
        1 & 3\\
        2 & 4
      \end{ytableau}\qquad
      \begin{ytableau}
        1 & 2
      \end{ytableau}\qquad \varnothing
  \end{align*}
\end{eg}

We record two facts about Richardson tableaux that we need later.
\begin{thm}[{\cite{KaPr25}}]
  The following hold.
  \begin{enumerate}
    \item Richardson tableaux are characterized by the fact that each evacuation slide is an \emph{$L$-slide} \cite[Theorem 5.3]{KaPr25}. That is, as one computes the evacuation tableau for $T$ a Richardson tableau, the path traced by the ``hole'' during every jeu-de-taquin slide is $L$-shaped.
    \item A tableau $T$ is Richardson if and only if $\evac{T}$ is Richardson \cite[Corollary 3.23]{KaPr25}.
  \end{enumerate}
  \label{th:richardson_evacuation}
\end{thm}

We make note of some crucial consequences of the first part of Theorem~\ref{th:richardson_evacuation}.
For $T$ a Richardson tableau, the tableau $T'$ obtained after a single evacuation slide, where we ignore the frozen cell, is also Richardson. 
Indeed the $L$-slide characterization immediately implies that all evacuation slides applied to $T'$ are $L$-slides.
Furthermore it is the case that 
\[
  \crop(\evac{T})=\evac{\crop(T)}.
\]

We now associate a pair of permutations to any SYT.  
It will turn out that this pair is well-aligned for Richardson tableaux.

\begin{defn}\label{def:richardsonpairs}
  Given $T\in \syt(\lambda)$, we define permutations $v_T,w_T\in S_n$ as follows:
  \begin{enumerate}
    \item $v_T^{-1}$ is the top down reading word of $\evac{T}$,
    \item $w_{\circ}w_T^{-1}w_{\circ}$ is the reading word of $T$.
  \end{enumerate}
\end{defn}
\begin{eg}\label{ex:vt_wt}
  Consider the $T$ and $\evac{T}$ from Example~\ref{ex:evacuation}. 
  Then
   $$\begin{tabular}{cccccccccccccc}
$v_T=$&1 & 6 & 9 & 2 & 7 & 11& 3 & 8& 4 & 10& 12 & 5\\
$w_T=$&11 & 6 & 1 & 9& 7 & 2 & 12& 3& 10& 8& 4& 5
\end{tabular}$$
\end{eg}

We now proceed to an alternative recursive description of $v_T$ and $w_T$ that will be more convenient for our purposes.
We begin with an elementary lemma about Richardson tableaux that relies on their characterization in terms of $L$-slides.

\begin{lem}
Let $\lambda\vdash n$ and $T\in \richtab(\lambda)$.
Suppose $m$ is the largest number such that the entries $\{1,2,\dots,m\}$ appear in the first column of $T$.
Then the first $m$ evacuation slides terminate in rows $r_1,\dots,r_m$ where $r_1>r_2>\cdots >r_m=1$. Consequently, the entries $\{n-m+1,\dots,n\}$ occupy a column strip in $\evac{T}$ with $n-m+1$ appearing at the end of the first row.
\end{lem}
\begin{proof}
  We shall repeatedly use the $L$-slide characterization of Richardson tableaux without explicit mention.
  The case $m=1$ is trivial.
  Suppose $m>1$.
  If $\lambda=(1^n)$ then the claim is clear. 
  So we assume that $\lambda_1>1$. By our condition on $m$ we know that $m+1$ appears in the first row next to $1$.
  In particular the $m$th evacuation slide must terminate at the end of the first row. 
  So $r_m=1$.

  To establish the claim we only need to show that $r_1>r_2$, as induction does the rest.
  If the $r_1$th row is of length $1$, again the claim is clear. So we assume that the $r_1$th row has length at least $2$.
  Let $j$ be the entry in the $r_1$th row and the second column. 
  Since $T$ is standard we know that $j$ is strictly larger than the entry immediately above it; call it $k$.
  After the first evacuation slide, $j$ moves to the first column whilst remaining in row $r_1$, whereas $k$ does not move. 
  This configuration guarantees that the next evacuation slide  terminates in a row strictly above $r_1$. This concludes the proof.
\end{proof}

\begin{eg}\label{ex:new_evacuation}
  Shown below is a Richardson tableau $T$ (left) and the result of applying the first three evacuation slides to it. The shaded entries belong to the evacuation tableau and are frozen. 
  In this case $m=3$, $r_1=4$, $r_2=2$, and $r_3=1$. 
  The entries $\{10,11,12\}$ occupy a column strip in $\evac{T}$ with $10$ appearing at the end of the first row. 
  \[
      \begin{ytableau}
        1 & 4 & 7 & 9 & 12\\
        2 & 5 & 8\\
        3 & 10\\
        6 & 11
      \end{ytableau}
      \qquad
      \begin{ytableau}
        1 & 3 & 6 & 8 & 11\\
        2 & 4 & 7\\
        5 & 9\\
        10 & *(red!50)12
      \end{ytableau}
      \qquad
      \begin{ytableau}
        1 & 2 & 5 & 7 & 10\\
        3 & 6 & *(red!50)11\\
        4 & 8\\
        9 & *(red!50)12
      \end{ytableau}
      \qquad
      \begin{ytableau}
        1 & 4 & 6 & 9 & *(red!50)10\\
        2 & 5 & *(red!50)11\\
        3 & 7\\
        8 & *(red!50)12
      \end{ytableau}
  \]
 
\end{eg}

Informally, the preceding lemma states that a column of entries $\{1,\dots,m\}$ on the top left in a Richardson tableau $T$ creates a distinguished column strip in the Richardson tableau $\evac{T}$.
To employ this observation recursively, we first show that there is a natural \emph{column strip decomposition} of any Richardson tableau completely determined by the entries in its first row.

\begin{defn}
For $T\in \syt(\lambda)$, we let $\First(T)=\{1=a_1<\dots<a_{\lambda_1}\}$ denote the entries in the first row of $T$ read from left to right. 
We further declare that $a_{\lambda_1+1}=|\lambda|+1$.
\end{defn}

\begin{lem}\label{le:column_strips_first_row}
  Let $T\in \richtab(\lambda)$.
  Let $\First(T)=\{1=a_1<\cdots<a_{\lambda_1}\}$.
   For $1\leq i\leq \lambda_1$, the cells occupied by the entries in $\{a_i,\dots,a_{i+1}-1\}$ form a column strip in $T$.
\end{lem}
\begin{proof}
This is a reformulation of \cite[Lemma 4.13(i)]{KaPr25} in the language of column strips.  
We include a short self-contained argument.

Let $\mathrm{Bad}$ denote the set of Richardson tableaux that do not have the desired property. 
Our goal is to show that $\mathrm{Bad} = \emptyset$.
We proceed by contradiction. 
Suppose $T \in \mathrm{Bad}$ is minimal, in the sense that $T$ has the fewest boxes among all elements of $\mathrm{Bad}$.
Let 
$X = \{a_i, \dots, a_{i+1}-1\}$
be the set of entries in $T$ that do not form a column strip, and let $S$ denote the skew shape consisting of these entries. 

We claim that $S$ must contain at least two boxes in the second row. 
Indeed, if this were not the case, then $\crop(T)$ would also be a Richardson tableau lying in $\mathrm{Bad}$, contradicting the minimality of $T$.
Let $j$ be the smallest element of $X$ that lies in the second row. 
Since $T$ is Richardson, we know that $j-1$ must appear in the first row.

Now consider the entry $j'$ immediately to the right of $j$ in $T$. 
We must have $j' > j+1$. 
Indeed, if $j' = j+1$, then by the definition of a Richardson tableau, $j$ would lie in the first row, a contradiction. 
Observe further that $j' \in X$. 
But then $j'-1$ belongs to both the first row and to $X$. 
This is impossible, since $X$ contains exactly one entry from the first row.
\end{proof}



Given $T\in \richtab(\lambda)$, we let \emph{$\Pi(T)$} denote the set partition of $\{1,\dots,|\lambda|\}$ whose blocks are given by the entries in the column strips from Lemma~\ref{le:column_strips_first_row}.
We note that $\Pi(T)$ is completely determined by $\First(T)$.
As a corollary of the preceding two lemmas we have the following.
\begin{cor}
\label{co:pi_evac}
  Fix $\lambda\vdash n$.
  For $T\in \richtab(\lambda)$, the partition $\Pi(\evac{T})$ is obtained by replacing the entries $i$ in each block of $\Pi(T)$ by $n+1-i$.
\end{cor}

\begin{eg}
Shown below are the column strip decompositions of $T$ and $\evac{T}$ from Example~\ref{ex:evacuation}.  
    \[
        \begin{ytableau}
        1 & *(red!50)2 & *(blue!50)5 & *(green!50)7 & *(orange!50)10\\
        *(red!50)3 & *(green!50)8 & *(orange!50)11\\
        *(red!50)4 & *(green!50)9\\
       *(blue!50)6 & *(orange!50)12
        \end{ytableau}\qquad\qquad\qquad 
        \begin{ytableau}
        *(orange!50)1 & *(green!50)4 & *(blue!50)7 & *(red!50)9 & 12\\
        *(orange!50)2 & *(green!50)5 & *(blue!50)8\\
        *(orange!50)3 & *(red!50)10\\
        *(green!50)6 & *(red!50)11
      \end{ytableau}
    \]
    For completeness, we record both $\Pi(T)$ and $\Pi(\evac{T})$ below.
    \begin{align*}
      \Pi(T)&=\{1\}\sqcup \{2,3,4\}\sqcup\{5,6\}\sqcup \{7,8,9\}\sqcup\{10,11,12\},\\
      \Pi(\evac{T})&=\{12\}\sqcup \{9,10,11\}\sqcup \{7,8\}\sqcup\{4,5,6\}\sqcup\{1,2,3\}.  
    \end{align*}  
\end{eg}

We record some additional consequences separately as they will be useful in proving that the pairs $(v_T,w_T)$ satisfy a stronger condition than being well-aligned.
We omit the proof as it amounts to unraveling Definition~\ref{def:richardsonpairs} which tells  us how to obtain $(v_T,w_T)$ from the appropriate reading words, together with the correspondence $\Pi(T)\leftrightarrow \Pi(\evac{T})$ of Corollary~\ref{co:pi_evac} needed to phrase the data for $w_T$ in terms of the blocks of $\evac{T}$.

\begin{cor}\label{co:positions_of_1_to_k_and_descents}
Fix $\lambda\vdash n$ and let $\lambda_1=k$.
For $T\in \richtab(\lambda)$, consider the column strip decomposition of $\evac{T}$. Suppose that $m_1$ through $m_k$ (resp. $M_1$ through $M_k$) are the smallest (resp. largest) entries within each block in increasing order.
\begin{enumerate}
  \item The positions of $1,2,\dots,k$ in $v_T$ are given by $m_1$ through $m_k$, and $$\des{v_T}= \{M_1,\dots,M_{k-1}\}.$$
  \item The positions of $1,2,\dots,k$ in $w_T$ are given by $M_1$ through $M_k$, and $$\des{w_T}= \{M_1,\dots,M_{k-1}, M_k\}^c.$$
\end{enumerate}
Since $M_k=n$, the descent sets $\des{v_T}$ and $\des{w_T}$ partition $[n-1]$.
\end{cor}
The parallel between parts (1) and (2) of Corollary~\ref{co:positions_of_1_to_k_and_descents} is a reflection of the fact that evacuation is an involution. The interested reader is also referred  to \cite[Proposition 6.4]{KaPr25} as an alternative means to deriving part (2) from part (1); the result in \textit{loc. cit.} tells us how  $w_T$ may be computed from $v_{\evac{T}}$.

\begin{defn}
  \label{de:very}
  We call $(v,w)\in\wellaligned{n}$ \emph{very well-aligned} if $(ww_{\circ},vw_{\circ})$ is well-aligned as well.
\end{defn}
Since right multiplication by $w_{\circ}$ reverses the one line notation of a permutation, we see that $(v,w)$ is very well-aligned if and only if the following hold:
\begin{enumerate}[label=(\roman*)]
        \item \label{item:very_well_aligned_1}$v^{-1}(1)\leq w^{-1}(1)$,
        \item \label{item:very_well_aligned_2}all indices $v^{-1}(1)\leq i\leq w^{-1}(1)-1$ are ascents in $v$, 
        \item \label{item:very_well_aligned_3}all indices $v^{-1}(1)\leq i\leq w^{-1}(1)-1$ are descents in $w$, and
        \item \label{item:very_well_aligned_4}$(\delta(v),\delta(w))$ is very well-aligned.
    \end{enumerate}

\begin{lem}\label{le:richardson_very_well_aligned}
  Fix $\lambda\vdash n$ and let $T\in \richtab(\lambda)$. 
  Then $(v_T,w_T)$ is very well-aligned. 
\end{lem}
\begin{proof}
  Let $k\coloneqq \lambda_1$.
  By Corollary~\ref{co:positions_of_1_to_k_and_descents}, we know that $v_T^{-1}(i)=m_i$ and $w_T^{-1}(i)=M_i$ for $1\leq i\leq k$, and $m_i\leq M_i$ for $1\leq i\leq k$.
  Since the blocks are intervals of consecutive integers, the range $[m_i,M_i-1]$ contains none of $M_1,\dots,M_k$; hence the descent-set equalities of Corollary~\ref{co:positions_of_1_to_k_and_descents} give that every such index is an ascent in $v_T$ and a descent in $w_T$.
  In particular, Conditions~\ref{item:very_well_aligned_1}--\ref{item:very_well_aligned_3} hold for $(\delta^i(v_T),\delta^i(w_T))$ all $0\leq i\leq k-1$. To establish the claim observe that 
  \[
    (\delta^k(v_T),\delta^k(w_T))=(v_{\crop(T)},w_{\crop(T)}).
  \]

  We briefly explain this equality. 
  Note that $v_{\crop(T)}$ can be read off from $\evac{\crop(T)}$, but this latter tableau equals $\crop(\evac{T})$. Cropping the first row of $\evac{T}$, whose entries correspond to the positions of $1$ through $k$ in $v_T$, and then standardizing the remaining entries to obtain an SYT then amounts to applying $\delta^{k}$ to $v_T$.
  Similarly, the positions of $1$ through $k$ in $w_T$, considered from the right, are determined by the first row of $T$. Cropping the first row of $T$ and then standardizing amounts to applying $\delta^{k}$ to $w_T$.
\end{proof}

\begin{eg}
  For the Richardson tableau $T$ in Example~\ref{ex:cropping}, we have $\crop(T)$ and $\evac{\crop(T)}$ equaling the following tableaux.
  \begin{align*}
    \begin{ytableau}
        1 & 4 & 6\\
        2 & 5\\
        3 & 7
      \end{ytableau}\qquad
      \begin{ytableau}
        1 & 3 & 5\\
        2 & 6\\
        4 & 7
      \end{ytableau}
  \end{align*}
  It follows that
   $$\begin{tabular}{cccccccccccccc}
$v_{\crop(T)}=$&1& 4& 2& 6&  3& 5& 7\\
$w_{\crop(T)}=$&6& 1& 4& 2& 7& 5& 3
\end{tabular}$$
  which the reader can check agree with $\delta^5(v_T)$ and $\delta^5(w_T)$ respectively from Example~\ref{ex:vt_wt}.
\end{eg}

\begin{eg}
We borrow the well-aligned pair $(v_T,w_T)=(1523467,7123654)$ from \cite[Example 10.9]{KaPr25} for ease of comparison, and use  Theorem~\ref{th:well_aligned_to_dominant} to compute the expansion
$[X^w_v]=\sum_{u\in S_n} c^{w}_{u,v}[X^u].$
We have $v_T^{\uparrow}=4512367$ and $v_T^{\uparrow}v_T^{-1}w_T=7412653$.
The Schubert polynomial indexed by $v_T$ is the Schur polynomial $s_{(3)}(x_1,x_2)$ whereas the Schubert polynomial indexed by $v_T^{\uparrow}$ is the dominant monomial $x_1^3x_2^3$.
For $w_{\circ}=n(n-1)\cdots 1$, by the identity $c^w_{u,v}=c^{w_{\circ}u}_{v,w_{\circ}w}$ whenever $u,v,w\in S_n$ we have
$$\schub{v}\schub{w_{\circ}w}=\sum_{u} c^{u}_{v,w_{\circ}w}\schub{u}=\sum_{u\in S_n}c^{w}_{w_{\circ}u,v}\schub{u}+\sum_{u\in S_{\infty}\setminus S_n}c^{u}_{v,w_{\circ}w}\schub{u},$$ 
and so applying this for $v=v_T^{\uparrow}$ and $w=v_T^{\uparrow}v_T^{-1}w_T$, if we restrict this sum to those $u\in S_n$ we can read off the well-aligned Schubert coefficients (after reindexing the permutations $u\mapsto w_{\circ}u$). From Figure~\ref{fig:richardson_schubert_example} it follows that
\begin{align*}
\schub{4512367}\cdot \schub{1476235}&=\schub{4765123}+\schub{5763124}+\schub{6735124}+\schub{6752134}+\cdots,
\end{align*}
where we omitted the Schuberts $\schub{u}$ with $u\not\in S_7$, 
so
$$[X^w_v]=[X^{4123765}]+[X^{3125764}]+[X^{2153764}]+[X^{2136754}].$$
\end{eg}

We note that prior to our work, the cases in which $[X_{v_T}^{w_T}]$ had been computed  combinatorially in the Schubert basis were limited to tableaux of hook shape \cite{Gu89}; see also the work of Graham--Zierau \cite{GZ11} which applies localization techniques but the expressions produced there are not combinatorial.
We briefly recall the combinatorial interpretation in \cite{Gu89}, recast in the language of \emph{G\"uemes tableaux} in \cite[\S 10]{KaPr25}.
Given a Richardson tableau $T$ of hook shape, the full expansion
involves counting certain semistandard tableaux of staircase shape whose (column) reading words give reduced words.
Our expansion, on the other hand, involves counting certain saturated chains in $k$-Bruhat order, for varying $k$.

\begin{prob}
  Find a bijection between the chains in $k$-Bruhat order that arise in this way for well-aligned pairs $(v_T,w_T)$ associated to a Richardson tableau $T$ of hook shape and G\"uemes tableaux.
\end{prob}
For $u=v_T$ for $T$ a Richardson tableau of hook shape of size $n$, the permutation $u^{\uparrow}$ is relatively straightforward to compute -- it is given by a shuffle of the letters $k,k+1,\dots,n$ and $1,2,\dots,k-1$ for some $k$ (depending on $T$), where the letters in the former appear in increasing order and those in the latter appear in decreasing order, and the first letter of $u^{\uparrow}$ is $k$.
It is easily checked that this is dominant.

\begin{figure}[!ht]
\includegraphics[scale=0.8]{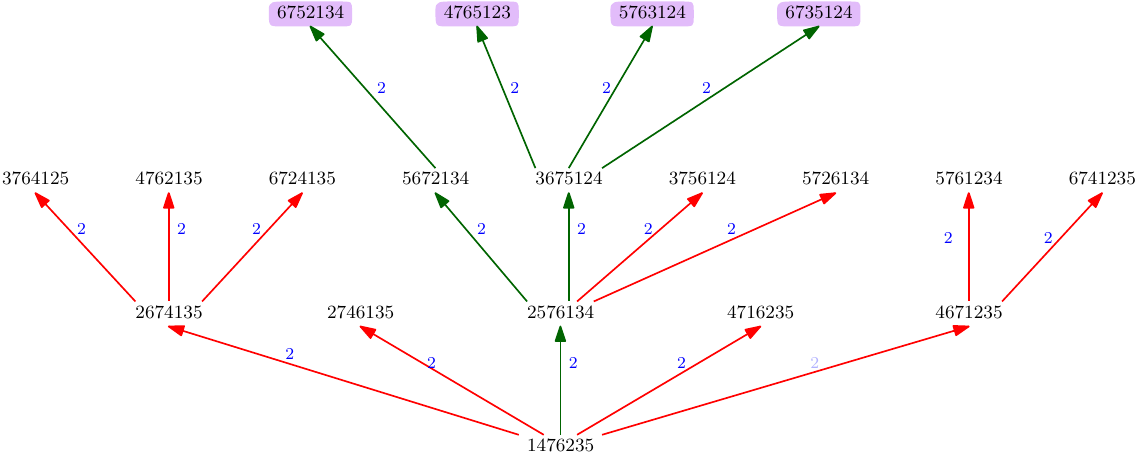}
\caption{The product $x_1^3x_2^3\schub{1476235}$ computed via Corollary~\ref{co:iterated_pieri}. The green chains in $2$-Bruhat order are the ones that contribute. At every level we have only recorded terms indexed by permutations in $S_7$.}
\label{fig:richardson_schubert_example}
\end{figure}

\section{The geometry of well and very-well aligned Richardson varieties}
\label{sec:class_expansions}
We refer the reader to \cite{AF24,BL00,BGP25,WY23} for detailed treatment of notions/results in Schubert calculus that we employ but do not define.
Throughout we work over $\CC$. We denote by $\GL_{n}$ the group of invertible $n\times n$ matrices.
Let $B$ and $B^-$  be the subsets of  $\GL_{n}$ comprising  upper triangular and lower triangular invertible $n\times n$ matrices respectively. 
A \emph{flag} in $\CC^n$ is a sequence of subspaces $V_1\subset V_2\subset \cdots \subset V_n=\CC^n$ where $\dim(V_i)=i$.
The \emph{complete flag variety} $\flag{n}$ is the set of all flags in $\CC^n$.
The group $\GL_n$ acts transitively on $\flag{n}$ via its natural action on $\CC^n$.
The stabilizer of the \emph{standard flag} $E_{\bullet}=E_1\subset E_2\subset \cdots \subset E_n=\CC^n$, where $E_i=\mathrm{span}\{e_1,\dots,e_i\}$, is $B$.
Thus we have $\flag{n}\cong \GL_n/B$.

For $v,w \in S_{n}$, we will denote the \emph{Schubert cells} in $\fl{n}$ by $\mathring{X}^w=BwB/B\cong \mathbb{A}^{\ell(w)}$ and \emph{opposite Schubert cells} by $\mathring{X}_v=B^-vB/B\cong \mathbb{A}^{\binom{n}{2}-\ell(v)}$. We will denote their closures the \emph{Schubert varieties} in $\fl{n}$ by $X^w=\overline{BwB/B}$, the \emph{opposite Schubert varieties} by $X_v=\overline{B^-vB/B}$, and for $v\le w$ the \emph{Richardson varieties} $X_v^w\coloneqq X^w\cap X_v=\overline{BwB/B}\cap \overline{B^-vB/B}$.
 By work of Borel \cite{Bor53}, we have an identification  $H^\bullet(\fl{n})=\ZZ[\xl_n]/\symide{n}$ where $\symide{n}$ is the ideal generated by positive degree symmetric polynomials in $x_1$ through $x_n$.
The classes $[X^w]\in H_\bullet(\fl{n})$ of the Schubert varieties for $w\in S_n$ form a homology basis Kronecker dual to the cohomology basis $\{\schub{v}(x_1,\ldots,x_n)\suchthat v\in S_n\}\subset H^{\bullet}(\flag{n})$ of Schubert polynomials \cite{LS82}, which themselves represent the Poincar\'e dual classes to the opposite Schubert varieties $X_v$ in $H^\bullet(\fl{n})$. 

We have an expansion $$[X^w_v]=\sum c^w_{u,v}[X^u]$$ where $c^w_{u,v}$ are the generalized Littlewood--Richardson coefficients appearing in  $\schub{u}\schub{v}=\sum c^w_{u,v}\schub{w}$. 
More generally, writing $\deg_{X^w_v}f$ for the degree $\int_{X^w_v}f$ of the top class of $f$ on $X^w_v$, we have
\[
\deg_{X^w_v}f=\ev_0\partial_w(\schub{v}f)=\Phi^w_vf.\]
Indeed, on the Schubert basis $\deg_{X^w_v}\schub{u}=\langle [X^w_v],\schub{u}\rangle=c^w_{u,v}=\ev_0\partial_w(\schub{v}\schub{u})$; the displayed identity then follows by linearity in $f$.
In particular,
$\deg_{X^w_v}\schub{u}=c^w_{u,v}$.


\subsection{Translation equivalence for well-aligned Richardson varieties}

Recall that the \emph{generalized Pl\"ucker functions} of $M\in \GL_n$ are given by
$$\plucker{w}(M)=\prod_{k=1}^n \det M^{1,\ldots,k}_{w(1),\ldots,w(k)}$$
where $M^A_B$ is the submatrix of $M$ determined by the columns from $A$ and the rows of $B$. 
We have $\plucker{w}(w'M)=\pm \plucker{(w')^{-1}w}(M)$,  $\plucker{w}(Bw'B)=0$ whenever $w\not\le_B w'$, and $\mathring{X}^w\subset \{\plucker{w}\ne 0\}$.
In particular, $w'\{\plucker{w}=0\}=\{\plucker{w'w}=0\}$ for all $w'\in S_n$.
The characterization that follows is well known; we include a proof for completeness.
\begin{lem}\label{le:plucker_vanishing}
    If $v\le_B w$ then we can write $$X^w_v=\bigcap_{u\not\in [v,w]}\{\plucker{u}=0\}.$$
\end{lem}
\begin{proof}
    It suffices to show the statement for $X^w$ and $X_v$. 
    We do this for $X^w$, the case for $X_v$ is similar. 
    We have to show that 
    \[
    X^w=\bigcap_{u\not\le_B w}\{\plucker{u}=0\}.
    \] 
    Recall the Bruhat decomposition $\GL_n/B=\bigsqcup_{u\in S_n}\mathring{X}^u$. 
    Note that $X^w$ and $\bigcap_{u\not\le_B w}\{\plucker{u}=0\}$ are both closed subvarieties containing $\mathring{X}^w$ (and hence containing the closure $X^w$), and disjoint from $\mathring{X}^u$ for $u\not\le_B w$. The claim now follows from the equality
    \[
    X^w=\bigcup_{u\le_B w}\mathring{X}^u.\qedhere
    \]
\end{proof}


\begin{prop}\label{pro:richardson_translate}
    If $[v,w]$ is translation-equivalent to $[v',w']$, then $X_{v'}^{w'}=v'v^{-1}X_v^w.$
\end{prop}
\begin{proof}
    Translation-equivalence gives $v'v^{-1}[v,w]=[v',w']$, so by Lemma~\ref{le:plucker_vanishing} and the $S_n$-equivariance of generalized Pl\"ucker functions up to sign we have \begin{equation*}
    v'v^{-1}X^w_v=\bigcap_{u\notin [v,w]}\{\plucker{v'v^{-1}u}=0\}=\bigcap_{u'\notin [v',w']}\{\plucker{u'}=0\}=X^{w'}_{v'}.\qedhere\end{equation*}
\end{proof}

\begin{cor}\label{co:well_aligned_translate}
    If $(v,w)\in \wellaligned{n}$, then $(v^{\uparrow}v^{-1})X^w_v=X^{v^{\uparrow}v^{-1}w}_{v^{\uparrow}}$. In particular, writing $v'=v^{\uparrow}$ and $w'=v^{\uparrow}v^{-1}w$, we have
    $[X^w_v]=[X^{w'}_{v'}]\in H_\bullet(\fl{n}).$
\end{cor}
\begin{proof}
    By Proposition~\ref{pr:to_dominant}(2), the intervals $[v,w]$ and $[v^{\uparrow},v^{\uparrow}v^{-1}w]$ are translation-equivalent, so the first assertion is the special case $(v',w')=(v^{\uparrow},v^{\uparrow}v^{-1}w)$ of the preceding proposition. Since left translation by a permutation matrix is an automorphism of $\fl{n}$ acting trivially on homology, the two Richardson varieties have equal classes.
\end{proof}

Taking homology classes in Proposition~\ref{pro:richardson_translate} proves Fact~\ref{fact:translation_coeff}: if $[v,w]$ and $[v',w']$ are translation-equivalent then $[X^w_v]=[X^{w'}_{v'}]$ in $H_\bullet(\fl{n})$, and comparing Schubert expansions yields $c^w_{u,v}=c^{w'}_{u,v'}$.


\begin{rem}\label{rem:kl}
Suppose $[v,w]$ and $[v',w']$ are translation-equivalent.
Because the equivalence is realized by the single left translation
$g = v'v^{-1}$, and $(gx)^{-1}(gy) = x^{-1}y$ for all $x,y$, the map $g$
carries the Bruhat graph of $[v,w]$ isomorphically onto that of
$[v',w']$, preserving the reflection labeling of edges exactly.
Since the Kazhdan--Lusztig $R$-polynomial $R(v,w)$ is computed from increasing paths in this labelled graph
with respect to any reflection ordering \cite{Dy93}, and every
increasing path from $v$ to $w$ stays inside $[v,w]$, we conclude
$R(v,w) = R(v',w')$: combinatorial invariance in the sense of
Dyer holds within each translation-equivalence class.
We thank an anonymous referee for this suggestion.
\end{rem}

\subsection{Building well-aligned Richardsons from pattern maps and geometric push-pull}

Let $\Psi_{1,i}:\GL_{n-1}/B\to \GL_n/B$ be the map which takes $MB$ to $M'B$ where $M'$ is obtained by including $M$ into an $n\times n$ matrix avoiding the first row and $i$th column, and then inserting a $1$ into the first row and $i$th column. For example
$$\Psi_{1,2}\begin{bmatrix}a&b\\c&d\end{bmatrix}=\begin{bmatrix}0&1&0\\a&0&b\\c&0&d\end{bmatrix}.$$
This is an inclusion map, and as shown in \cite[Theorem 4.4]{nst_c} we have 
$$
\Psi_{1,i}X_v^w=X^{\ins_i(w)}_{\ins_i(v)}.
$$
Furthermore, we have $\Psi_i^*:H^\bullet(\fl{n})\to H^\bullet(\fl{n-1})$ is given by $\rope{i}$, so for $f(x_1,\ldots,x_n)\in H^\bullet(\fl{n})$ we have
$$\deg_{X^{\ins_i(w)}_{\ins_i(v)}}f=\deg_{X^w_v}\rope{i}f.$$
Let $\pi_i:\GL_n/B\to \GL_n/P_i$ be the projection map where $P_i$ is the minimal parabolic subgroup associated to $s_i$. 
Then by the classical computation of Bernstein--Gelfand--Gelfand \cite{BGG73} we have $(\pi_i)^*(\pi_i)_*f=\partial_i f$. 
This push-pull identity is explained in several places; see \cite[Chapter 3]{Man01} and \cite[Chapter 10, Lemma 6.5]{AF24}.
We claim that if $i$ is an ascent of both $v$ and $w$ we have $$\pi_i^{-1}\pi_iX^w_v=X^{ws_i}_{v},$$ with $\pi_i|_{X^w_v}$ generically injective. 
Indeed, we have $\pi_i^{-1}\pi_iX^w\subset X^{ws_i}$ and $\pi_i^{-1}\pi_iX_v\subset X_v$ which shows the left hand side is contained in the right hand side, and the identity $\Phi^{ws_i}_v=\Phi^w_v\partial_i$, verified in the proof of \Cref{lem:Phisimplify} implies that $(\pi_i^*)(\pi_i)_*[X^w_v]=[X_v^{ws_i}]$. In particular, we have
$$\deg_{X^{ws_i}_v}f=\deg_{X^w_v}\partial_if.$$

Therefore because a well-aligned pair can be obtained by successively applying either $(v,w)\mapsto (\ins_iv,\ins_iw)$ and $(v,w)\mapsto (v,ws_i)$ with $i$ an ascent of both $v,w$, we see that $X^w_v$ can be obtained as a successive application of $\Psi_{1,i}$ and $\pi_i^{-1}\pi_i$. Furthermore,
$\Phi^w_vf=\deg_{X^w_v}f,$
so the derivation of $\Phi^w_v$ as a successive application of operations $\rope{i}$ and $\partial_i$ (namely the one recorded in Lemma~\ref{lem:Phisimplify}) follows geometrically from the corresponding degree map facts above.
\subsection{Smoothness and very well-alignedness}

We revisit another result of Karp--Precup in light of ours and offer a generalization.
In \cite[\S 9]{KaPr25} they establish that the Richardson variety $X_{v_T}^{w_T}$ is smooth for all Richardson tableaux $T$.
We show that this smoothness property extends to Richardson varieties coming from very well-aligned pairs.
In this subsection, for a variety $Y$ we write $\operatorname{Sing}Y\subseteq Y$
for its singular locus, i.e. the set of points at which $Y$ is not smooth.
Our proof also uses \cite[Corollary 9.5]{KaPr25},  a consequence of \cite[Corollary 2.10]{BC12} which gives a criterion for smoothness of Richardson varieties by reducing it to testing smoothness at two special $T$-fixed points; we record the precise form we need, together with its short deduction, in \Cref{fact:KaPr25}.

\begin{fact}
\label{fact:KaPr25}
$X_v^w$ is smooth if and only if it is smooth at the points $vB,wB\in X_v^w$.
\end{fact}
\begin{proof}
    \cite[Corollary 9.5]{KaPr25} states that  $X_v^w$ is smooth if and only if $X_v$ is smooth at $wB$ and $X^w$ is smooth at $vB$. But \cite[Corollary 2.10]{BC12} identifies \[
  \operatorname{Sing}X^w_v
  =\bigl(\operatorname{Sing}X^w\cap X_v\bigr)\cup\bigl(X^w\cap\operatorname{Sing}X_v\bigr)
\]
so the result follows as $vB,wB$ lie in the dense open affine charts $\mathring{X}_v\subset X_v$ and $\mathring{X}^w\subset X^w$ so lie outside of the singular loci of $X_v$ and $X^w$ respectively.
\end{proof}

We will also need an extension of Deodhar's criterion \cite{D85} for smoothness at $T$-fixed points of Schubert varieties to arbitrary Richardson varieties. Recall that we denote $t_{ab}\in S_n$ for the transposition swapping $a$ and $b$, where by convention we always set $a<b$.

\begin{prop}
\label{prop:Deodhar}
    For $uB\in X_v^w$, we have
    $$\#\{t_{ab}\in S_n\suchthat u\,t_{ab}\in [v,w]\}\ge \ell(w)-\ell(v)$$
    with equality if and only if $X^w_v$ is smooth at $u\in X^w_v$.
\end{prop}
\begin{proof}
As $u\in [v,w]$ the numerical inequality is established in {\cite{D85}}. Furthermore it was shown for $w=w_{\circ}$ or $v=\idem$ that equality holds if and only if $u$ is smooth in the corresponding Richardson variety (which is a Schubert or opposite Schubert variety). We now establish the equality case in general. 
    In \cite[Proof of Theorem 1.1]{KWY13} there is an isomorphism of an open affine chart around $u\in X^w_v$ given by
    \begin{equation}
\label{eqn:uvw}u\mathring{X}_{\idem}\cap X^w_v\cong (X_v\cap \mathring{X}^u) \times (X^w\cap \mathring{X}_u)\end{equation}
    where $u\mapsto (u,u)$. Applying \eqref{eqn:uvw} with $(u;v,w_{\circ})$ in place of $(u;v,w)$ shows $u\mathring{X}_{\idem}\cap X_v\cong (X_v\cap \mathring{X}^u)\times \mathbb{A}^{\binom{n}{2}-\ell(u)}$ and applying \eqref{eqn:uvw} with $(u;\idem,w)$ in place of $(u;v,w)$  shows $u\mathring{X}_{\idem}\cap X^w\cong (X^w\cap \mathring{X}_u)\times \mathbb{A}^{\ell(u)}$. Hence $u\in X^w_v$ is smooth if and only if $u$ is smooth in $X^w$ and in $X_v$. Deodhar's result \cite{D85} says these happen if and only if we have both of \begin{align*}\#\{t_{ab}\in S_n\suchthat ut_{ab}\le w\}&=\ell(w)\text{ or equivalently }\#\{t_{ab}\in S_n\suchthat ut_{ab}\in [u,w]\}=\ell(w)-\ell(u),\text{ and }\\
    \#\{t_{ab}\in S_n\suchthat ut_{ab}\ge v\}&=\binom{n}{2}-\ell(v)\text{ or equivalently }\#\{t_{ab}\in S_n\suchthat ut_{ab}\in [v,u]\}=\ell(u)-\ell(v).
    \end{align*}

    Because $$\#\{t_{ab}\in S_n\suchthat ut_{ab}\in [v,w]\}=\#\{t_{ab}\in S_n\suchthat ut_{ab}\in [v,u]\}+\#\{t_{ab}\in S_n\suchthat ut_{ab}\in [u,w]\}$$
    and $\ell(w)-\ell(v)=(\ell(w)-\ell(u))+(\ell(u)-\ell(v))$, we have the numerical equality for the triple $(u;v,w)$ if and only if we have equality for both triples $(u;v,u)$ and $(u;u,w)$, which we have just established is equivalent  to $u\in X^w_v$ being smooth.
    
\end{proof}

\begin{lem}
\label{lem:domsmooth}
    If $v$ is dominant and $v<_B w$ then $X^w_v$ is smooth at $wB$.
\end{lem}
\begin{proof}
    The boxes in the Rothe diagram $\rothe(w)$ are in bijection with the inversions $t_{ab}$ of $w$, i.e. with $wt_{ab}\le_B w$, by taking such a transposition to the box at position $(w(b),a)$. Furthermore the $\ell(v)$-many boxes in the Rothe diagram $\rothe(v)$ are inversions $t_{ab}$ such that $wt_{ab}$ has its permutation matrix occupy a box in $\rothe(v)$, and hence $v\not\le_B wt_{ab}$ by \eqref{eqn:rotheBruhat}. We conclude by \Cref{prop:Deodhar} that the remaining $(\ell(w)-\ell(v))$-many inversions of $w$ have $v\le_B wt_{ab}$ and $X^w_v$ is smooth at $wB$. 
\end{proof}

\begin{rem}
    Geometrically, if $v$ is a dominant permutation then $X_v\subset \fl{n}$ is the subvariety of those $gB$ where the entries of $g$ inside the Rothe diagram of $v$ are set to zero (this follows from the rank condition characterization of Schubert varieties and is the key observation in the study of Ding partition varieties, see e.g. \cite{D97,DMR07}). 
    In the Bruhat decomposition $\GL_n/B=\bigsqcup_{u\in S_n}\mathring{X}^u$, we have $\mathring{X}^u\cong \mathbb{A}^{\ell(u)}$ as the set of matrices which are $1$ in the entries $(u(i),i)$, have indeterminate entries in the boxes of the Rothe diagram $\rothe(u)$, and zeroes elsewhere. Therefore we have
    $$X_v^w=\bigsqcup_{u\le_B w}X_v\cap \mathring{X}^u =\bigsqcup_{u\in [v,w]}X_v\cap \mathring{X}^u$$
    and $X_v\cap \mathring{X}^u\cong \mathbb{A}^{\ell(u)-\ell(v)}$ is the coordinate subspace of $\mathring{X}^u\cong \mathbb{A}^{\ell(u)}$ where we set the entries in $\rothe(v)$ to $0$ (this is well-defined by \eqref{eqn:rotheBruhat}). This gives an affine paving of the Richardson varieties $X^w_v$ with $v$ dominant, and because $X_v\cap \mathring{X}^w$ is a dense open chart around $wB$ isomorphic to $\mathbb{A}^{\ell(w)-\ell(v)}$ this  gives an alternate geometric way to verify that $wB\in X^w_v$ is smooth.
\end{rem}

\begin{thm}\label{th:smoothness}
  Let $(v,w)\in \wellaligned{n}$ be very well-aligned. Then the Richardson variety $X_v^w$ is smooth.
\end{thm}
\begin{proof}
By \Cref{pr:to_dominant} we may choose $[v',w']$ translation-equivalent to $[v,w]$ with $v'$ dominant. Then $X^w_v=v(v')^{-1}X^{w'}_{v'}$ by \Cref{pro:richardson_translate}, and so $X^{w}_{v}$ is smooth at $wB$ since $X^{w'}_{v'}$ is smooth at $w'B$ by \Cref{lem:domsmooth}.

    Since $(v,w)$ is very well-aligned, the pair $(ww_{\circ},vw_{\circ})$ is also well-aligned, so by \Cref{pr:to_dominant} the interval $[ww_{\circ},vw_{\circ}]$ is translation-equivalent to $[w''w_{\circ},v''w_{\circ}]$ with $w''w_{\circ}$ a dominant permutation. Exactly as in the previous paragraph, \Cref{lem:domsmooth} shows that $X^{v''w_{\circ}}_{w''w_{\circ}}$ is smooth at $v''w_{\circ}B$, and then \Cref{pro:richardson_translate} shows that $X^{vw_{\circ}}_{ww_{\circ}}$ is smooth at $vw_{\circ}B$. Now, right multiplication by $w_{\circ}$ is an anti-isomorphism from $[ww_{\circ},vw_{\circ}]$ to $[v,w]$. So, by \Cref{prop:Deodhar}, we get that $X^{vw_{\circ}}_{ww_{\circ}}$ is smooth at $vw_{\circ}B$ if and only if $X^w_v$ is smooth at $vB$. Because $X^w_v$ is smooth at $vB$ and $wB$, we conclude by \Cref{fact:KaPr25}.
\end{proof}


\section{Odds and ends}
\label{sec:misc}

To conclude this article we collect some numerological observations that may be of interest.
One may ask for quantitative information about how many Schubert structure coefficients have we computed. Put differently, one may inquire about the number $|\wellaligned{n}|$ of well-aligned pairs for a given nonnegative integer $n$. We offer the following conjecture.

\begin{conj}
  Let $\mc{W}(x)$ denote the exponential generating function:
  \[
    \mc{W}(x)=\sum_{n\geq 0}|\wellaligned{n}|\frac{x^{n}}{n!}=1+1\frac{x^1}{1!}+3\frac{x^2}{2!}+17\frac{x^3}{3!}+147\frac{x^4}{4!}+1729\frac{x^5}{5!}+25827\frac{x^6}{6!}+468593\frac{x^7}{7!}+\cdots.
  \]
  Then $\mc{W}(x)$ satisfies the following functional equation:
  \[
    \mc{W}'(x)=\frac{\mc{W}(x)^2}{2-\mc{W}(x)}.
  \]
\end{conj}
The sequence in question matches \cite[A234289]{oeis} for $0\leq n\leq 7$. 
Both sides of the functional equation can be assigned combinatorial meaning easily, which in turn suggests there is a recursive decomposition of ``pointed'' well-aligned pairs that witnesses the functional equation. 

As seen before, the well-aligned pairs in $\wellaligned{n}^{132}$ are crucial for us.
The cardinality of this set equals $(2n-1)!!$; see Remark~\ref{re:double_factorial}.
Nevertheless there is a curious observation to be made.
It is the case that there exist $(v,w), (v',w')\in \wellaligned{n}^{132}$ such that the Bruhat intervals $[v,w]$ and $[v',w']$ are translation-equivalent.
So one may quotient $\wellaligned{n}^{132}$ further by this equivalence and inquire about the resulting number of equivalence classes.
We make the following conjecture.
\begin{conj}
    The total number of equivalence classes of translation-equivalent intervals is given by \cite[A111088]{oeis}, which counts circular planar electrical networks \cite{ALT15}.
\end{conj}

\bibliographystyle{hplain}
\bibliography{main.bib}

\end{document}